\documentclass[reqno]{amsart}
\usepackage{amsmath, amsthm, amssymb, amstext}
\usepackage{hyperref}
\usepackage[dvipsnames]{xcolor}
\hypersetup{pdfborder={0 0 0}, colorlinks}
\usepackage{enumitem}
\allowdisplaybreaks
\usepackage{todonotes}

\newtheorem{theorem}{Theorem}
\newtheorem{remark}[theorem]{Remark}
\newtheorem{lemma}[theorem]{Lemma}
\newtheorem{proposition}[theorem]{Proposition}

\newcommand{\R}{{\mathbb R}}
\newcommand{\N}{{\mathbb N}}

\numberwithin{theorem}{section}
\numberwithin{equation}{section}

\title [Critical double phase problems]{Critical double phase problems involving sandwich-type nonlinearities}

\author[C. Farkas]{Csaba Farkas}
\address[C. Farkas]{Sapientia Hungarian University of Transylvania, Department of Mathematics and Computer Science, T\^{a}rgu Mure\textcommabelow{s}, Romania \& Corvinus Centre for Operations Research, Corvinus Institute for Advanced Studies, Corvinus University of Budapest, F\H ov\'am t\'er 8, 1093, Budapest, Hungary}
\email{farkascs@ms.sapientia.ro}

\author[A. Fiscella]{Alessio Fiscella}
\address[A. Fiscella]{Departamento de Matem\'atica, Universidade Estadual de Campinas, IMECC, Rua S\'ergio Buarque de Holanda 651, CEP 13083--859,  Campinas SP, Brazil}
\email{fiscella@unicamp.br}

\author[K. Ho]{Ky Ho}
\address[K. Ho]{Department of Mathematics and Statistics, University of Economics Ho Chi Minh City, 59C, Nguyen Dinh Chieu Street, Ho Chi Minh City, Vietnam}
\email{kyhn@ueh.edu.vn}

\author[P. Winkert]{Patrick Winkert}
\address[P. Winkert]{Technische Universit\"{a}t Berlin, Institut f\"{u}r Mathematik, Stra\ss{}e des 17.\,Juni 136, 10623 Berlin, Germany}
\email{winkert@math.tu-berlin.de}

\subjclass{35B33, 35J20, 35J25, 35J62, 35J70, 46E35, 47J10}
\keywords{concentration-compactness principle, critical growth, Krasnoselskii's genus theory, sandwich-type growth, unbalanced growth}

\begin{document}

\begin{abstract}
	In this paper we study problems with critical and sandwich-type growth represented by
	\begin{equation*}
		\begin{aligned}
			-\operatorname{div}\Big(|\nabla u|^{p-2}\nabla u + a(x)|\nabla u|^{q-2}\nabla u\Big)& = \lambda w(x)|u|^{s-2}u+\theta B\left(x,u\right)  && \text{in } \Omega,\\
			u& = 0 &&\text{on } \partial \Omega,
		\end{aligned}
	\end{equation*}
	where $\Omega\subset\mathbb{R}^N$ is a bounded domain with  Lipschitz boundary $\partial\Omega$, $1<p<s<q<N$, $\frac{q}{p}<1+\frac{1}{N}$, $0\leq a(\cdot)\in  C^{0,1}(\overline{\Omega})$, $\lambda$, $\theta$ are real parameters, $w$ is a suitable weight and $B\colon \overline{\Omega}\times \mathbb{R}\to\mathbb{R}$ is given by
	\begin{align*}
		B(x,t) :=b_0(x)|t|^{p^*-2}t+b(x)|t|^{q^*-2}t,
	\end{align*}
	where $r^*:=Nr/(N-r)$ for $r\in\{p,q\}$. Here the right-hand side combines the effect of a critical term given by $B(\cdot,\cdot)$ and a sandwich-type perturbation with exponent $s \in (p,q)$. Under different values of the parameters $\lambda$ and $\theta$, we prove the existence and multiplicity of solutions to the problem above. For this, we mainly exploit different variational methods combined with topological tools, like a new concentration-compactness principle, a suitable truncation argument and the Krasnoselskii's genus theory, by considering very mild assumptions on the data $a(\cdot)$, $b_0(\cdot)$ and $b(\cdot)$.
\end{abstract}

\maketitle

\section{Introduction and main results}

In the last decade, the double phase operator has gained interest in many different research areas. This operator is defined by
\begin{align}\label{double-phase-operator}
	\operatorname{div}\Big(|\nabla u|^{p-2}\nabla u + a(x)|\nabla u|^{q-2}\nabla u\Big),\quad 1<p<q,
\end{align}
and  arises from the study of general reaction-diffusion equations with nonhomogeneous diffusion and transport aspects. Applications can be found in biophysics, plasma physics and chemical reactions, with double phase features, where the function $u$ corresponds to the concentration term, and the differential operator represents the diffusion coefficient. The related integral functional to  \eqref{double-phase-operator} has the form
\begin{align}\label{double-phase-integral}
	J(u)=\int_\Omega \bigg( \frac{|\nabla u|^{p}}p + a(x) \frac{|\nabla u|^{q}}q \bigg) \,\mathrm{d} x,
\end{align}
for a bounded domain $\Omega\subset\R^N$. It appeared for the first time in the work of Zhikov \cite{Zhikov-1986} and is useful in the context of homogenization and elasticity theory. In this setting, the coefficient $a(\cdot)$ is associated to the geometry of composites made of two materials of hardness $p$ and $q$. Functionals of the form \eqref{double-phase-integral} can be regarded as particular instances of the seminal contributions by Marcellini \cite{Marcellini-1991, Marcellini-1989} which address issues concerning nonstandard growth and $(p,q)$-growth conditions. In fact, the regularity theory in \cite{Marcellini-1991} also applies to double phase integrals. In this regard, we also cite the recent works on $u$-dependence by Cupini--Marcellini--Mascolo \cite{Cupini-Marcellini-Mascolo-2023} and Marcellini \cite{Marcellini-2023}. For further reading on this topic, we also recommend reading the paper by Marcellini \cite{Marcellini-2021} which presents recent results on problems with nonstandard growth. Subsequent to this, the regularity results obtained by Marcellini for the special case of the double phase setting have been refined by a series of papers by Baroni--Colombo--Mingione \cite{Baroni-Colombo-Mingione-2015,Baroni-Colombo-Mingione-2016,Baroni-Colombo-Mingione-2018} and Colombo--Mingione \cite{Colombo-Mingione-2015a, Colombo-Mingione-2015b}. In contrast to \cite{Marcellini-1991} in which $a(\cdot)$ must be Lipschitz for the double phase setting, the papers by Baroni, Colombo  and Mingione only require H\"{o}lder continuity of the weight function $a(\cdot)$. As previously indicated, double phase problems appear in various applications. We refer to the papers by Bahrouni--R\u{a}dulescu--Repov\v{s} \cite{Bahrouni-Radulescu-Repovs-2019} on transonic flows, Benci--D’Avenia--Fortunato--Pisani \cite{Benci-DAvenia-Fortunato-Pisani-2000} on quantum physics, Cherfils--Il'yasov \cite{Cherfils-Ilyasov-2005} for reaction diffusion systems and Zhikov \cite{Zhikov-1995, Zhikov-2011} on the Lavrentiev gap phenomenon, the thermistor problem and the duality theory.

In this paper we examine the existence and multiplicity of solutions to problems with critical and sandwich-type growth represented by
\begin{align} \label{Eq.Sandwich}
	- \operatorname{div}A\left(x,\nabla u\right) = \lambda w(x)|u|^{s-2}u+\theta B\left(x,u\right) \quad \text{in } \Omega,  \quad u=0 \quad \ \text{on } \partial \Omega,
\end{align}
where $\Omega\subset\mathbb{R}^N$ is a bounded domain with  Lipschitz boundary $\partial\Omega$, $\lambda$, $\theta$ are real parameters, $w$ is a suitable weight, while $A\colon \overline{\Omega}\times \mathbb{R}^N \to \mathbb{R}^N$ and $B\colon \overline{\Omega}\times \mathbb{R}\to\mathbb{R}$ are given by
\begin{align}\label{struttura}
	A(x,\xi):=|\xi|^{p-2}\xi+a(x)|\xi|^{q-2}\xi,\quad
	B(x,t) :=b_0(x)|t|^{p^*-2}t+b(x)|t|^{q^*-2}t,
\end{align}
where $r^*:=Nr/(N-r)$ for $r\in\{p,q\}$. Denoting
\begin{align*}
	\Omega_+:=\{x\in\Omega\colon\,\, a(x)>0\},
\end{align*}
we suppose the following structure conditions on the data of problem \eqref{Eq.Sandwich}:
\begin{enumerate}[label=\textnormal{(H$_\arabic*$)},ref=\textnormal{H$_\arabic*$}]
	\item\label{H1}
		$1<p<s<q<N$, $\dfrac{q}{p}<1+\dfrac{1}{N}$, $0\leq a(\cdot)\in  C^{0,1}(\overline{\Omega})$ and $\Omega_+\ne \emptyset$.
	\item\label{H2}
		$0<b_0(\cdot)\in L^\infty(\Omega)$ and $0\leq b(\cdot)\in L^\infty(\Omega)$ such that $b(x)\leq C a(x)^{\frac{q^*}{q}}$ for a.a.\,$x\in\Omega$ and
		\begin{align}\label{C.I}
			\big\|b^{\frac{1}{q^*}}u\big\|_{q^*}\leq C\big\|a^{\frac{1}{q}}\nabla u\big\|_q,\quad\text{for any } u\in C_c^\infty(\Omega),
		\end{align}
		with some $C>0$.
	\item\label{H3}
		$w\colon\Omega\to\R$ is a measurable function such that  $\left|\left\{x\in\Omega_+\colon\,\, w(x)>0 \right\}\right|>0$, $w\chi_{\{b=0\}}b_0^{-\frac{s}{p^*}}\in L^{\frac{p^*}{p^*-s}}(\Omega)$,   $w\chi_{\{b>0\}}b^{-\frac{s}{q^*}}\in L^{\frac{q^*}{q^*-s}}(\Omega)$ and
		\begin{align}\label{Cond-H3}
			\int_{\Omega}w(x)|u|^s\,\mathrm{d} x\leq C_w\left(\int_{\Omega}a(x)|\nabla u|^q\,\mathrm{d} x\right)^{\frac{s}{q}},
		\end{align}
		for any $u\in C_c^\infty(\Omega)$ with some $C_w>0$. Here, $\chi_E$ denotes the characteristic function of $E$ and  $\chi_{\{b>0\}}b^{-\frac{s}{q^*}}:=0$ on the set $\{b=0\}$.
\end{enumerate}

\begin{remark}
When $\Omega_+=\emptyset$, problem \eqref{Eq.Sandwich} under hypotheses \eqref{H1}--\eqref{H3} reduces to a superlinear $p$-Laplace problem with critical growth, which was studied by Ho--Sim \cite{Ho-Sim-2023} in the setting of a generalized $p(\cdot)$-Laplacian. Note that if $\operatorname{supp}\,(b)\subset \Omega_+$ as in the paper by Colasuonno--Perera \cite{Colasuonno-Perera-2025}, then \eqref{C.I} in condition \eqref{H2} holds true, see \cite[Proposition 4.11]{Colasuonno-Perera-2025}. We also point out that condition \eqref{Cond-H3} in \eqref{H3} is satisfied if $b(x)>0$ for a.a.\,$x\in\Omega$ and $wb^{-\frac{s}{q^*}}\in L^{\frac{q^*}{q^*-s}}(\Omega)$. Indeed, by H\"{o}lder's inequality and \eqref{H2} we have for $u\in C_c^\infty(\Omega)$
	\begin{align*}
		\int_\Omega w(x)|u|^s\,\mathrm{d}x
		\leq \left\|wb^{-\frac{s}{q^*}} \right\|_{\frac{q^*}{q^*-s}} \left\|b^{\frac{1}{q}}u\right\|_{q^*}^s
		\leq C^s \left\|wb^{-\frac{s}{q^*}} \right\|_{\frac{q^*}{q^*-s}} \left(\int_\Omega a(x)|\nabla u|^q\,\mathrm{d}x\right)^{\frac{s}{q}}.
	\end{align*}
\end{remark}

The main feature of problem \eqref{Eq.Sandwich} is the combination of the double phase operator with a right-hand side which consists of a sandwich-type nonlinearity $t \mapsto \lambda w(x)|t|^{s-2}t$ with an indefinite weight $w(\cdot)$ and exponent $s \in (p,q)$, along with a critical growth term $B(\cdot,\cdot)$ given in \eqref{struttura}. Note that the solutions of \eqref{Eq.Sandwich} shall belong to the Musielak-Orlicz Sobolev space $W_0^{1,\mathcal H}(\Omega)$ which arises from the generalized $N$-function $\mathcal{H}\colon \overline{\Omega}\times [0,\infty)\to  [0,\infty)$ given by
\begin{align*}
	\mathcal{H}(x,t):= t^p+a(x)t^q\quad \text{for }(x,t)\in\overline{\Omega}\times[0,\infty).
\end{align*}
From this definition, we see that the double phase operator given in \eqref{double-phase-operator} is a generalization of the $p$-Laplacian and of the $(p,q)$-Laplacian for $p<q$, by setting $a(\cdot)\equiv0$ or $\inf a(\cdot)>0$, respectively. Here, we point out that the critical term $\mathcal{B}(\cdot,\cdot)$ defined by
\begin{align*}
	\mathcal{B}(x,t):=b_0(x)|t|^{p^*} +b(x)|t|^{q^*}\quad \text{for }(x,t)\in \overline{\Omega}\times \R
\end{align*}
is also of double phase type and appears to have the natural critical growth in relation to the operator. However, this term further complicates the study of \eqref{Eq.Sandwich}. Indeed, in our variational approach we have to overcome the lack of compactness of the embedding $W_0^{1,\mathcal H}(\Omega)\hookrightarrow L^{\mathcal B}(\Omega)$. In order to do so we exploit a new concentration and compactness argument inspired by the work of Ha--Ho \cite[Theorem 2.1]{Ha-Ho-2024} taking care of the Luxemburg norms of the Musielak-Orlicz spaces $W_0^{1,\mathcal H}(\Omega)$ and $L^{\mathcal B}(\Omega)$. In this direction, we provide a suitable compactness threshold for the energy functional associated to \eqref{Eq.Sandwich}, which allows us to handle the sandwich-type perturbation with exponent $s\in(p,q)$ by \eqref{H1}. This sandwich-type situation for \eqref{Eq.Sandwich} is more interesting and delicate, since it is strictly related to the double phase growth of the main operator in \eqref{Eq.Sandwich}. In fact, when $a(\cdot)\equiv0$ in \eqref{Eq.Sandwich} and \eqref{struttura}, the sandwich case with $s\in(p,q)$ cannot occur.

Recently, problem \eqref{Eq.Sandwich} has been studied in the papers by Colasuonno--Perera \cite{Colasuonno-Perera-2025} and Farkas--Fiscella--Winkert \cite{Farkas-Fiscella-Winkert-2022} in case $\theta=1$. In \cite[Theorem 1.1]{Farkas-Fiscella-Winkert-2022}, the authors covered the sublinear situation with $s\in(1,p)$, proving the existence of infinitely many solutions of \eqref{Eq.Sandwich} with negative energy. For this, they exploited the Krasnoselskii's genus theory combined with a truncation argument, by assuming very mild hypotheses for the terms in \eqref{struttura}. On the other hand, in \cite[Theorems 2.1 and 2.6]{Colasuonno-Perera-2025} the authors considered \eqref{Eq.Sandwich} in the regime $s\in[p,q^*)$, employing a Br\'{e}zis–Nirenberg-type approach \cite{Brezis-Nirenberg-1983} to establish the existence of a mountain-pass solution. They first showed that there exists a threshold $\beta^*>0$ such that, for any $\beta\in (0,\beta^*)$, every $\textup{(PS)}_\beta$ sequence for the energy functional  admits a subsequence that converges weakly to a nontrivial weak solution of problem \eqref{Eq.Sandwich}. Subsequently, they proved that the critical mountain‑pass energy level lies below this threshold. For this, they required very restrictive assumptions on the dimension $N$, the exponent $p$ and the weights $a(\cdot)$ and $b(\cdot)$ of \eqref{struttura}, see \cite[Theorems 2.1 and 2.6]{Colasuonno-Perera-2025}. In particular, they needed that either $a(\cdot)\equiv0$ on a suitable ball $B_r(x_0)$, or $a(\cdot)\equiv a_0>0$ and $b(\cdot)\equiv b_\infty>0$ are constant on $B_r(x_0)$.

Motivated by the papers of Colasuonno--Perera \cite{Colasuonno-Perera-2025} and Farkas--Fiscella--Winkert \cite{Farkas-Fiscella-Winkert-2022}, we want to provide existence and multiplicity results for solutions of \eqref{Eq.Sandwich} with negative energies under the sandwich case $s\in(p,q)$. In order to state our first main result, we note that the assumption \eqref{H3} implies that
\begin{align*}
	\mathcal{W}_+:=\left\{\phi\in W^{1,\mathcal{H}}_0(\Omega)\colon\,\, \operatorname{supp}\,(\phi)\subset\Omega_+  \text{ and }  \int_\Omega w(x)\phi_+^s\,\mathrm{d}x>0\right\}\neq\emptyset,
\end{align*}
where $\phi_+:= \max\{\phi,0\}$, see Kawohl--Lucia--Prashanth \cite[Proposition 4.2]{Kawohl-Lucia-Prashanth-2007}. We can therefore define
\begin{align}\label{la*}
	\lambda_0:=\inf_{\phi\in\mathcal{W}_+}\, C_0\left(\frac{\int_\Omega a(x)|\nabla \phi|^q\,\mathrm{d} x}{q}\right)^{\frac{s-p}{q-p}}\left(\frac{\int_\Omega |\nabla \phi|^p\,\mathrm{d}x}{p}\right)^{\frac{q-s}{q-p}}\frac{s}{\int_\Omega w(x)\phi_+^s\,\mathrm{d}x},
\end{align}
where $C_0=C_0(p,q,s)$ is given as
\begin{align}\label{C0}
	C_0:=\left(\frac{q-p}{s-p}\right)^{\frac{s-p}{q-p}}\left(\frac{q-p}{q-s}\right)^{\frac{q-s}{q-p}}.
\end{align}

Our first result reads as follows.

\begin{theorem}\label{Theo.NS}
	Let hypotheses \eqref{H1}--\eqref{H3} be satisfied. Then,  for any given $\lambda>\lambda_0$ with $\lambda_0$ as in \eqref{la*}, there exists $\theta_\ast=\theta_\ast(\lambda)>0$ such that for any $\theta\in \left(0,\theta_\ast\right)$, problem~\eqref{Eq.Sandwich} has a nontrivial nonnegative solution with negative energy.
\end{theorem}

Theorem \ref{Theo.NS} generalizes the $(p,q)$-Laplacian situation studied by Ho--Sim \cite[Theorem 1.1]{Ho-Sim-2021} in a nontrivial way. Indeed, in \cite[Theorem 1.1]{Ho-Sim-2021} they presented their problem in the trivial intersection space $W_0^{1,p}(\Omega)\cap W_0^{1,q}(\Omega)=W_0^{1,q}(\Omega)$ with $p<q$. In this way, they minimized their energy functional on the ball
\begin{align*}
	B_r=\left\{u\in W_0^{1,q}(\Omega)\colon\,\, \|\nabla u\|_q\leq r\right\},
\end{align*}
disregarding the norm $\|\nabla u\|_p$. In Theorem \ref{Theo.NS}, we still apply a minimization argument, combined with Ekeland's variational principle, but taking care of the Luxemburg type norm of the Musielak-Orlicz Sobolev space $W_0^{1,\mathcal H}(\Omega)$. For this, we need a suitable compactness threshold for the validity of the Palais-Smale condition of our energy functional.

Our second result is devoted to the multiplicity of solutions to problem \eqref{Eq.Sandwich} stated in the following theorem.

\begin{theorem}\label{Theo.MS}
	Let hypotheses \eqref{H1}--\eqref{H3} be satisfied and assume that there exists a ball $B\subset\Omega_+$ such that $w(x)>0$ for a.a.\,$x\in B$. Then, there exists $\{\theta_j\}_{j\in\N}$ with $0<\theta_j<\theta_{j+1}$, such that for any $j\in\N$ and with $\theta\in \left(0,\theta_j\right)$, there exist $\lambda_\star$, $\lambda^\star>0$  with $\lambda_\star<\lambda^\star$ and possibly depending on $\theta$, such that for any $\lambda\in\left(\lambda_\star,\lambda^\star\right)$, problem \eqref{Eq.Sandwich} admits at least $j$ pairs of distinct solutions with negative energy.
\end{theorem}

The proof of Theorem \ref{Theo.MS} relies on a careful combination of variational and topological tools, such as truncation techniques and Krasnoselskii's genus theory, similar to the work by Farkas--Fiscella--Winkert \cite[Theorem 1.1]{Farkas-Fiscella-Winkert-2022} under the sublinear situation. However, the sandwich perturbation with exponent $s\in(p,q)$ does not allow us to provide the existence of infinitely many solutions for \eqref{Eq.Sandwich}. Indeed, as in \cite[Theorem 1.1]{Farkas-Fiscella-Winkert-2022}, we can construct a monotone and non-decreasing sequence $\{c_j\}_{j\in\mathbb{N}}$ of critical values by the genus theory. However, we can only guarantee that the values $c_1\leq c_2\leq\ldots\leq c_j$ are negative when $\theta<\theta_j$ and $\lambda>\lambda_*$, with possibly $\theta_j\to0$ and $\lambda_*=\lambda_*(\theta_j)\to\infty$ as $j\to\infty$. This is the crucial difference with respect to the sublinear case in Theorem 1.1 of \cite{Farkas-Fiscella-Winkert-2022}, where the $\{c_j\}_{j\in\N}$ are all negative when $\lambda<\lambda^*$ is sufficiently small.

We emphasize that Theorem \ref{Theo.MS} completes the picture of the paper by Farkas--Fiscella--Winkert \cite[Theorem 1.1]{Farkas-Fiscella-Winkert-2022} and generalizes the multiplicity results obtained by Baldelli--Brizi--Filippucci  \cite[Theorem 1]{Baldelli-Brizi-Filippucci-2021} and by Ho--Sim \cite[Theorem 1.3]{Ho-Sim-2023}. Indeed, in \cite[Theorem 1]{Baldelli-Brizi-Filippucci-2021} they dealt with a $(p,q)$-Laplacian situation, i.e., with $\inf a(\cdot)>0$ while in \cite[Theorem 1.3]{Ho-Sim-2023} they considered a more general operator which involves two sides, one given by the $p(\cdot)$-Laplacian and the other one given an operator set on a suitable ball $B$. Then, in \cite[Theorem 1.3]{Ho-Sim-2023} they cover a sandwich-type case with $s<p^-:=\min p(x)$ and $s>q_B$, where $q_B$ is the exponent of the second side. Thus, they do not cover a truly sandwich case for the $p(\cdot)$-Laplacian, i.e.\,with $\min p(x)=:p^-<s<p^+:=\max p(x)$. In Theorem \ref{Theo.MS}, even working with Luxemburg norms, we can cover a complete sandwich perturbation with $s\in(p,q)$ in \eqref{Eq.Sandwich}. Finally, we mention related papers dealing with critical growth for double phase problems, see the works by Arora--Fiscella--Mukherjee--Winkert \cite{Arora-Fiscella-Mukherjee-Winkert-2022}, Farkas--Winkert \cite{Farkas-Winkert-2021}, Feng--Bai \cite{Feng-Bai-2024}, Ho--Kim--Zhang \cite{Ho-Kim-Zhang-2024}, Kumar--R\u{a}dulescu--Sreenadh \cite{Kumar-Radulescu-Sreenadh-2020}, Liu--Papageorgiou \cite{Liu-Papageorgiou-2021}, Papageorgiou--Vetro--Winkert \cite{Papageorgiou-Vetro-Winkert-2024,Papageorgiou-Vetro-Winkert-2023}, and Papageorgiou--Zhang \cite{Papageorgiou-Zhang-2020}, see also the paper by Ho--Perera--Sim \cite{Ho-Perera-Sim-2023} on the Br\'ezis-Nirenberg problem for the $(p,q)$-Laplacian. Note that the methods and techniques used in these papers are different from the ones applied in our work.

The paper is organized as follows. In Section \ref{Section_2}, we introduce the notation used along the paper and some technical properties of the Musielak-Orlicz spaces. In Section \ref{Section_3}, we prove the compactness property of the energy functional related to \eqref{Eq.Sandwich} while in Sections \ref{Section_4} and \ref{Section_5}, we prove Theorems \ref{Theo.NS} and \ref{Theo.MS}, respectively.

\section{Preliminaries and Notation}\label{Section_2}

In this section, we recall the main properties of Musielak-Orlicz and Musielak-Orlicz Sobolev spaces as well as of the double phase operator. Most of the results are taken from the papers by Colasuonno--Squassina \cite{Colasuonno-Squassina-2016}, Crespo-Blanco--Gasi\'nski--Harjulehto--Winkert \cite{Crespo-Blanco-Gasinski-Harjulehto-Winkert-2022}, Ho--Winkert \cite{Ho-Winkert-2023} and Liu--Dai \cite{Liu-Dai-2018}, see also the monographs by Chlebicka--Gwiazda--\'{S}wierczewska-Gwiazda--Wr\'{o}blewska-Kami\'{n}ska \cite{Chlebicka-Gwiazda-Swierczewska-Gwiazda-Wroblewska-Kaminska-2021},  Diening--Harjulehto--H\"{a}st\"{o}--R\r{u}\v{z}i\v{c}ka \cite{Diening-Harjulehto-Hasto-Ruzicka-2011}, Harjulehto--H\"{a}st\"{o} \cite{Harjulehto-Hasto-2019}, and Papageorgiou--Winkert \cite{Papageorgiou-Winkert-2024}.

First, we want to introduce the underlying generalized $N$-functions describing the behavior of the function $A\colon \overline{\Omega}\times \mathbb{R}^N \to \mathbb{R}^N$ and $B\colon \overline{\Omega}\times \mathbb{R}\to\mathbb{R}$ given in \eqref{struttura}. For this purpose, let $1<\alpha<\beta$, $0< c(\cdot) \in L^{1}(\Omega)$ and $0 \leq d(\cdot) \in L^{1}(\Omega)$. We define the $N$-function $\Phi\colon \overline{\Omega}\times  [0,\infty)\to [0,\infty)$ as
\begin{align*}
	\Phi(x,t):=c(x)t^{\alpha}+d(x)t^{\beta}
	\quad\text{for } (x,t)\in \overline{\Omega}\times [0,\infty),
\end{align*}
while the associated modular $\rho_{\Phi}$ to $\Phi$  is given by
\begin{align}\label{modular-Lp}
	\rho_{\Phi}(u):= \int_{\Omega} \Phi (x,|u|)\,\mathrm{d}x.
\end{align}
Denoting by $M(\Omega)$ the set of all measurable functions on $\Omega$,  the corresponding Musielak-Orlicz space $L^{\Phi}(\Omega)$ is defined
by
\begin{align*}
	L^{\Phi}(\Omega):=\left \{u\in M(\Omega)\colon\,\, \rho_{\Phi}(u) < \infty \right\},
\end{align*}
endowed with the Luxemburg norm
\begin{align*}
	\|u\|_{\Phi}:= \inf \left \{ {\tau} >0 \colon\,\, \rho_{\Phi}\left(\frac{u}{\tau}\right) \leq 1  \right \}.
\end{align*}

The following proposition gives the relation between the modular $\rho_{\Phi}(\cdot)$ and its norm $\|\cdot\|_{\Phi}$, see  Crespo-Blanco--Gasi\'nski--Harjulehto--Winkert \cite[Proposition 2.13]{Crespo-Blanco-Gasinski-Harjulehto-Winkert-2022} for a detailed proof.

\begin{proposition}\label{P.m-n}
	Let $1<\alpha<\beta$, $0< c(\cdot) \in L^{1}(\Omega)$, $0 \leq d(\cdot) \in L^{1}(\Omega)$, $\lambda>0$, and $u\in L^{\Phi}(\Omega)$ while $\rho_{\Phi}(\cdot)$ is as in \eqref{modular-Lp}. Then, the following hold:
	\begin{enumerate}
		\item[\textnormal{(i)}]
			If $u\neq 0$, then $\|u\|_{\Phi}=\lambda$ if and only if $ \rho_{\Phi}(\frac{u}{\lambda})=1$.
		\item[\textnormal{(ii)}]
			$\|u\|_{\Phi}<1$ (resp.\,$>1$, $=1$) if and only if $ \rho_{\Phi}(u)<1$ (resp.\,$>1$, $=1$).
		\item[\textnormal{(iii)}]
			If $\|u\|_{\Phi}<1$, then $\|u\|_{\Phi}^{\beta}\leqslant \rho_{\Phi}(u)\leqslant\|u\|_{\Phi}^{\alpha}$.
		\item[\textnormal{(iv)}]
			If $\|u\|_{\Phi}>1$, then $\|u\|_{\Phi}^{\alpha}\leqslant \rho_{\Phi}(u)\leqslant\|u\|_{\Phi}^{\beta}$.
		\item[\textnormal{(v)}]
			$\|u\|_{\Phi}\to 0$ if and only if $\rho_{\Phi}(u)\to 0$.
		\item[\textnormal{(vi)}]
			$\|u\|_{\Phi}\to \infty$ if and only if $\rho_{\Phi}(u)\to \infty$.
	\end{enumerate}
\end{proposition}

Next, we assume that \eqref{H1} and \eqref{H2} are fulfilled, so that we can set
\begin{align*}
	\begin{aligned}
		\mathcal{H}(x,t) & :=t^{p} +a(x)t^{q} && \text{for } (x,t)\in \overline{\Omega}\times [0,\infty), \\
		\mathcal{B}(x,t) & :=b_0(x)|t|^{p^*} +b(x)|t|^{q^*} && \text{for } (x,t)\in \overline{\Omega}\times \R.
	\end{aligned}
\end{align*}
We define the Musielak-Orlicz Sobolev space $W^{1,\mathcal{H}}(\Omega)$ as
\begin{align*}
	W^{1,\mathcal{H}}(\Omega)
	=\left \{u \in L^{\mathcal{H}}(\Omega)\colon\,\, |\nabla u| \in L^{\mathcal{H}}(\Omega) \right \}
\end{align*}
equipped with the norm
\begin{align*}
	\|u\|_{1,\mathcal{H}} = \|u\|_{\mathcal{H}}+\|\nabla u\|_{\mathcal{H}},
\end{align*}
where $\|\nabla u\|_{\mathcal{H}}=\| \, |\nabla u| \,\|_{\mathcal{H}}$. Furthermore, we denote by $W^{1,\mathcal{H}}_0(\Omega)$ the completion of $C^\infty_c(\Omega)$ in $W^{1,\mathcal{H}}(\Omega)$. In view of Colasuonno--Squassina \cite[Proposition 2.14]{Colasuonno-Squassina-2016}, we know that $L^{\mathcal{H}}(\Omega)$, $W^{1,\mathcal{H}}(\Omega)$ and $W^{1,\mathcal{H}}_0(\Omega)$ are separable and reflexive Banach spaces. In addition, the Poincar\'e inequality, namely
\begin{align*}
	\|u\|_{\mathcal{H}} \leq C\|\nabla u\|_{\mathcal{H}}\quad\text{for any
	} u \in W^{1,\mathcal{H}}_0(\Omega),
\end{align*}
holds true, see Colasuonno--Squassina \cite[Proposition 2.18 (iv)]{Colasuonno-Squassina-2016} or Crespo-Blanco--Gasi\'nski--Harjulehto--Winkert \cite[Proposition 2.19]{Crespo-Blanco-Gasinski-Harjulehto-Winkert-2022} under the weaker assumption $q<p^*$ instead of $\frac{q}{p}<1+\frac{1}{N}$. Based on this, we can equip the space $W^{1,\mathcal{H}}_0(\Omega)$ with the equivalent norm
\begin{align*}
	\|\cdot\|:=\|\nabla \cdot\|_{\mathcal{H}}.
\end{align*}
We have the following embedding results, see Colasuonno--Squassina \cite[Proposition 2.15]{Colasuonno-Squassina-2016}.

\begin{proposition}\label{prop-embeddings}
	Let hypothesis \eqref{H1} be satisfied. Then, the following hold:
	\begin{enumerate}
		\item[\textnormal{(i)}]
		$W_{0}^{1,\mathcal{H}}(\Omega)\hookrightarrow W_{0}^{1,p}(\Omega)$ is continuous;
		\item[\textnormal{(ii)}]
		$W_{0}^{1,\mathcal{H}}(\Omega)\hookrightarrow L^{p^{*}}(\Omega)$ is continuous.
		\item[\textnormal{(iii)}]
		$W_{0}^{1,\mathcal{H}}(\Omega)\hookrightarrow L^{r}(\Omega)$ is continuous and compact for all $1\leq r<p^*$.
	\end{enumerate}
\end{proposition}

The next proposition can be found in the work by Ho--Winkert \cite[Proposition 3.7]{Ho-Winkert-2023} and plays a key role to handle the critical Sobolev term in problem \eqref{Eq.Sandwich}.

\begin{proposition}\label{prop_CE1}
	Let hypothesis \eqref{H1} be satisfied and
	\begin{align*}
		\mathcal{G}(x,t):=|t|^{k}+a(x)^{\frac{m}{q}}|t|^{m}	\quad\text{for } (x,t)\in \overline{\Omega}\times \R,
	\end{align*}
	where $1\leq k\leq p^*$ and $1\leq m\leq q^*$. Then, we have the continuous embedding
	\begin{align}\label{Prop-S-E}
		W^{1,\mathcal{H}}(\Omega)
		\hookrightarrow  L^{\mathcal{G}}(\Omega).
	\end{align}
	Furthermore, if $k<p^*$ and $m< q^*$, then the embedding in \eqref{Prop-S-E} is compact. In particular, it holds
	\begin{align*}
		W^{1,\mathcal{H}}(\Omega)
		\hookrightarrow L^{\mathcal{H}}(\Omega) \quad\text{compactly}.
	\end{align*}
\end{proposition}

Let us define the operator $L\colon W^{1,\mathcal H}_0(\Omega)\to \left(W^{1,\mathcal H}_0(\Omega)\right)^*$ by
\begin{align}\label{operator-def}
	\langle L(u),v\rangle:=\int_\Omega\left(|\nabla u|^{p-2}+a(x)|\nabla u|^{q-2}\right)\nabla u\cdot\nabla v\,\mathrm{d}x
\end{align}
for any $u$, $v\in W^{1,\mathcal H}_0(\Omega)$, where $\left(W^{1,\mathcal H}_0(\Omega)\right)^*$ denotes the dual space of $W^{1,\mathcal H}_0(\Omega)$ and $\langle \,\cdot \,,\cdot\,\rangle$ is the related duality pairing. The following result is taken from Liu--Dai \cite[Proposition 3.1 (ii)]{Liu-Dai-2018}.

\begin{proposition}\label{P2.4}
	Let hypothesis \eqref{H1} be satisfied. Then, the mapping $L\colon W^{1,\mathcal H}_0(\Omega)$ $\to \left(W^{1,\mathcal H}_0(\Omega)\right)^*$ given in \eqref{operator-def} is of type $(\operatorname{S}_+)$, that is, if $u_n\rightharpoonup u$ in $W^{1,\mathcal H}_0(\Omega)$ and $\limsup\limits_{n\to\infty}\langle L(u_n)-L(u),u_n-u\rangle\leq0$, then $u_n\to u$ in $W^{1,\mathcal H}_0(\Omega)$.
\end{proposition}

The next compactness results is needed for the sandwich perturbation. It can be proved in a similar way as Lemma 4.1 by Ho--Kim--Sim \cite{Ho-Kim-Sim-2019} via Vitali's convergence theorem.

\begin{proposition}\label{prop.CE2}
	Let hypothesis \eqref{H3} be satisfied and let $\{u_n\}_{n\in\N}\subseteq W_0^{1,\mathcal{H}}(\Omega)$ be a sequence with $u_n \rightharpoonup  u$. Then, it holds
	\begin{align*}
		\int_\Omega w(x)|u_n|^s\,\mathrm{d}x\to \int_\Omega w(x)|u|^s\,\mathrm{d} x
		\quad\text{and}\quad
		\int_\Omega |w(x)||u_n-u|^s\,\mathrm{d} x\to 0
	\end{align*}
	as $n\to\infty$.
\end{proposition}

Next, we want to recall further notations which will be used in the sequel. First, we mainly work with terms set as
\begin{align*}
	\|u\|_{d,m}&:=\left(\int_{\Omega}d(x)| u| ^{m}\,\mathrm{d} x\right)^{1/m},\quad&
	\|u\|_{m}&:=\left(\int_{\Omega}| u| ^{m}\,\mathrm{d} x\right)^{1/m},\\
	\|u\|_{L^m(d,E)}&:=\left(\int_{E}d(x)| u| ^{m}\,\mathrm{d} x\right)^{1/m},\quad
	&\|u\|_{L^m(E)}&:=\left(\int_{E}| u| ^{m}\,\mathrm{d} x\right)^{1/m},
\end{align*}
while $|E|$ indicates the Lebesgue measure of a measurable set $E\subset\R^N$. Also, for any $r>0$, we denote the sets
\begin{align*}
	B(y,r)&:=\{x\in\R^N\colon\,\,  |x-y|< r\},\\
	B_r&:=\{u\in W_0^{1,\mathcal{H}}(\Omega)\colon\,\, \|u\|<r\},\\
	\partial B_r&:=\{u\in W_0^{1,\mathcal{H}}(\Omega)\colon\,\, \|u\|=r\}.
\end{align*}
Meanwhile, in the next section, we denote by $\mathcal{M}(\overline{\Omega})$ the space of the Radon measures on $\overline{\Omega}$.

Considering the notation above, note that under assumption \eqref{H2} we have
\begin{align*}
	S_p:=\underset{\phi\in C_c^\infty(\Omega)\setminus\{0\}}{\inf}
	\frac{\|\nabla\phi\|_p}{\| \phi \|_{b_0,p^*}}>0
	\quad\text{and}\quad
	S_q:=\underset{\phi\in C_c^\infty(\Omega)\setminus\{0\}}{\inf}
	\frac{\|\nabla\phi\|_{a,q}}{\| \phi \|_{b,q^*}}>0\quad \text{if}\ b\not\equiv 0.
\end{align*}
Thus, there exists a constant $C_e>1$ such that
\begin{align}
	\label{Ce}
	\|u\|_{\mathcal{B}}\leq C_e\|u\|,\quad \|u\|_{b_0,p^*}\leq C_e\|\nabla u\|_p\quad\text{and}\quad \|u\|_{b,q^*}\leq C_e\|\nabla u\|_{a,q}
\end{align}
for any $u\in W_0^{1,\mathcal{H}}(\Omega)$.

In order to determine (nonnegative) solutions of problem \eqref{Eq.Sandwich}, we introduce the energy functionals $J$, $J_+\colon W_0^{1,\mathcal{H}}(\Omega)\to\mathbb{R}$ given by
\begin{align*}
	J(u)   & :=\int_\Omega\mathcal A\left(x,\nabla u\right)\,\mathrm{d} x-\frac{\lambda}{s}\int_\Omega w(x)|u|^{s}\,\mathrm{d} x-\theta\int_\Omega\widehat{B}\left(x,u\right)\,\mathrm{d} x,\\
	J_+(u) & :=\int_\Omega\mathcal A\left(x,\nabla u\right)\,\mathrm{d} x-\frac{\lambda}{s}\int_\Omega w(x)u_+^{s}\,\mathrm{d} x-\theta\int_\Omega\widehat{B}\left(x,u_+\right)\,\mathrm{d} x,
\end{align*}
where $u_+:=\max\{u,0\}$ and
\begin{equation*}
	\begin{aligned}
	\mathcal{A}(x,\xi) & :=\frac{1}{p}|\xi|^{p}+\frac{a(x)}{q}|\xi|^{q}&\text{for }  &(x,\xi)\in \overline{\Omega}\times\R^N, \\
	\widehat{B}(x,t) & :=\frac{1}{p^*}b_0(x)|t|^{p^*}+\frac{1}{q^*}b(x)|t|^{q^*} & \text{for } &(x,t)\in \overline{\Omega}\times \R.
\end{aligned}
\end{equation*}
It is clear that $J$, $J_+$ are of class $C^1(W_0^{1,\mathcal{H}}(\Omega),\mathbb{R})$ and a critical point of $J$ (resp.\,$J_+$) is a solution (resp.\,nonnegative solution) to problem~\eqref{Eq.Sandwich}.

\section{A compactness result}\label{Section_3}

In this section we prove an important lemma which provides a compactness result regarding the functionals $J$ and $J_+$. For this, we recall that a sequence $\{u_n\}_{n\in\N}\subset W_0^{1,\mathcal{H}}(\Omega)$ is a Palais-Smale sequence for a functional $I\in C^1(W_0^{1,\mathcal{H}}(\Omega),\R)$ at level $c\in\R$ ($\textup{(PS)}_c$ sequence for short), if
\begin{align}\label{ps}
	I(u_n)\to c
	\quad\text{and}\quad
	I'(u_n)\to 0\quad\text{ in } \left(W_0^{1,\mathcal{H}}(\Omega)\right)^* \text{ as } n\to\infty.
\end{align}
We say that $I$ satisfies the Palais-Smale condition at level $c\in\R$ ($\textup{(PS)}_c$ condition for short), if any $\textup{(PS)}_c$ sequence admits a convergent subsequence in $W_0^{1,\mathcal{H}}(\Omega)$.

With a new concentration-compactness principle inspired by the work of Ha--Ho \cite[Theorem 2.1]{Ha-Ho-2024}, we study $\textup{(PS)}_c$ sequences under an important threshold defined below in \eqref{PS_c2}. Such a threshold can be derived by applying a concentration–compactness principle in the space $W_0^{1,\mathcal{H}}(\Omega)$. However, applying \cite[Theorem 2.1]{Ha-Ho-2024} would yield a threshold involving powers of $\theta$ and $\lambda$ that depend on both $p$ and $q$, which would complicate later estimates due to the presence of the intermediate exponent $s\in(p,q)$. For this reason, we establish a new concentration–compactness principle to obtain the desired threshold in  \eqref{PS_c2}.

\begin{lemma}\label{Le.PS}
	Let hypotheses \eqref{H1}--\eqref{H3} be satisfied and  let $\lambda$, $\theta>0$. Then, any bounded $\textup{(PS)}_c$ sequence of the functionals $J$ and $J_+$ admits a convergent subsequence in $W_0^{1,\mathcal{H}}(\Omega)$, provided that
	\begin{align}\label{PS_c2}
		c<C_1\min\left\{\theta^{-\frac{p}{p^*-p}},\theta^{-\frac{q}{q^*-q}}\right\}-C_2\lambda^{\frac{q}{q-s}}-C_3\lambda^{\frac{q^\ast}{q^\ast-s}}\theta^{-\frac{s}{q^\ast-s}},
	\end{align}
	where $C_1:=C_1(N,p,q)$, $C_2:=C_2(N,p,q,s,w)$ and $C_3:=C_3(N,p,q,w,b)$ are three suitable positive constants.
	If $b\equiv 0$, then any bounded $\textup{(PS)}_c$ sequence of the functionals $J$ and $J_+$ admits a convergent subsequence in $W_0^{1,\mathcal{H}}(\Omega)$, provided that
	\begin{align}\label{PS_c2'}
		c<C_1\theta^{-\frac{p}{p^*-p}}-C_2\lambda^{\frac{q}{q-s}}.
	\end{align}
\end{lemma}

\begin{proof}
	We only prove the assertion for $J$ since the situation for $J_+$ can be proved similarly. Also, we only consider the case  $b\not\equiv 0$, as the case $b\equiv 0$ is similar and easier to show.

	Fix $\lambda > 0$, $\theta>0$ and let $c\in\mathbb{R}$ satisfy \eqref{PS_c2} with $C_1$, $C_2$ and $C_3$ to be specified later. Let $\{u_n\}_{n\in\mathbb{N}}\subseteq W_0^{1,\mathcal{H}}(\Omega)$ be a bounded $\textup{(PS)}_c$ sequence of the functional $J$. Taking Proposition \ref{prop_CE1} into account, there exists $u\in W_0^{1,\mathcal{H}}(\Omega)$ such that up to a subsequence, if necessary not relabeled, we have
	\begin{equation}\label{1}
		\begin{aligned}
			u_n&\rightharpoonup u\quad\text{in }W_0^{1,\mathcal{H}}(\Omega),\quad u_n \to u\quad\text{in }L^{\mathcal H}(\Omega),\\
			u_n(x) &\to u(x)  \quad\text{for a.a.\,} x\in\Omega,
		\end{aligned}
	\end{equation}
	as $n\to\infty$.
	Furthermore, by virtue of Fonseca–Leoni \cite[Proposition 1.202]{Fonseca-Leoni-2007}, we can find bounded Radon measures $\mu$, $\nu$, $\overline{\mu}$, $\overline{\nu}\in \mathcal{M}(\overline{\Omega})$ such that
	\begin{equation}\label{12}
		\begin{aligned}
			|\nabla u_n|^{p}& \overset{\ast }{\rightharpoonup }\mu\quad\text{in } \mathcal{M}(\overline{\Omega}), \quad
			& b_0(x)|u_n|^{p^\ast}                                                                                   & \overset{\ast }{\rightharpoonup }\nu \quad\text{in } \mathcal{M}(\overline{\Omega}), \\
			a(x)|\nabla u_n|^{q} & \overset{\ast }{\rightharpoonup }\overline{\mu}\quad\text{in } \mathcal{M}(\overline{\Omega}),
			& b(x)|u_n|^{q^\ast} & \overset{\ast }{\rightharpoonup }\overline{\nu} \quad\text{in }\mathcal{M}(\overline{\Omega}),
		\end{aligned}
	\end{equation}
	as $n\to\infty$. Since $W_0^{1,\mathcal{H}}(\Omega)\hookrightarrow W_0^{1,p}(\Omega)$ by Proposition \ref{prop-embeddings} (i), we can apply the concentration-compactness principle by Lions \cite[Lemma I.1]{Lions-1985}, considering $b_0(\cdot)\in L^\infty(\Omega)$ by \eqref{H2}. Besides, thanks to \eqref{H1} and \eqref{H2} we can argue as in Ha--Ho \cite[Theorem 2.1]{Ha-Ho-2024} so that, taking also Lions \cite[Lemma I.1]{Lions-1985} into account, there exist five families of distinct points $\{x_i\}_{i\in\mathcal{I}}\subset\overline{\Omega}$, and of nonnegative numbers $\{\nu_i\}_{i\in\mathcal{I}}$, $\{\mu_i\}_{i\in\mathcal{I}}$, $\{\overline{\nu}_i\}_{i\in\mathcal{I}}$, $\{\overline{\mu}_i\}_{i\in\mathcal{I}}$, with $\mathcal{I}$ being an at most countable index set, such that we have
	\begin{align}\label{PL.S.m-n1}
		\mu \geq |\nabla u|^{p} + \sum_{i\in \mathcal{I}} \mu_i \delta_{x_i},\quad
		\nu=b_0(x)|u|^{p^\ast} + \sum_{i\in \mathcal{I}}\nu_i\delta_{x_i},\quad
		S_p \nu_i^{1/p^\ast} \leq \mu_i^{1/p},
	\end{align}
	for any $i\in \mathcal{I}$ and
	\begin{align}\label{PL.S.m-n2}
		\overline{\mu} \geq a(x)|\nabla u|^{q} + \sum_{i\in \mathcal{I}} \overline{\mu}_i \delta_{x_i},\quad
		\overline{\nu}=b(x)|u|^{q^\ast} + \sum_{i\in \mathcal{I}}\overline{\nu}_i\delta_{x_i} ,\quad
		S_q \overline{\nu}_i^{1/q^\ast} \leq \overline{\mu}_i^{1/q},
	\end{align}
	for any $i\in \mathcal{I}$, where $\delta_{x_i}$ is the Dirac mass at $x_i$. Then, we proceed by steps.

\vskip5pt
	\noindent{\bf Step 1:} $\mu_i+\overline{\mu}_i\leq \theta\left(\nu_i+\overline{\nu}_i\right)$ for any $i\in\mathcal I$.
\vskip5pt

	To this end, fix $i\in\mathcal I$ and let $\phi$ be in $C_c^\infty(\mathbb{R}^N)$ such that
	\begin{align*}
		0\leq \phi\leq 1,\quad \phi\equiv 1\quad \text{on }B\left(0,\frac{1}{2}\right)\quad \text{and}\quad \phi\equiv 0\quad \text{outside of } B(0,1).
	\end{align*}
	For $\varepsilon>0$, set $\phi_{i,\varepsilon}(x):= \phi(\frac{x-x_i}{\varepsilon})$ for $x\in\mathbb{R}^N$. Fix such an $\varepsilon$ and note that
	\begin{equation}\label{HH}
		\begin{aligned}
			&\int_{\Omega}\phi_{i,\varepsilon}|\nabla u_n|^p\,\mathrm{d} x+\int_{\Omega}\phi_{i,\varepsilon}a(x)|\nabla u_n|^q\,\mathrm{d} x\\
			& =\theta\int_\Omega\phi_{i,\varepsilon}b_0(x) |u_n|^{p^*}\,\mathrm{d} x+\theta\int_\Omega\phi_{i,\varepsilon}b(x)|u_n|^{q^*}\,\mathrm{d} x \\
			& \quad+\lambda\int_{\Omega}\phi_{i,\varepsilon}w(x)|u_n|^s\,\mathrm{d} x -\int_\Omega A(x,\nabla u_n)\cdot \nabla \phi_{i,\varepsilon}u_n \,\mathrm{d} x  +\langle J'(u_n) ,\phi_{i,\varepsilon}u_n \rangle.
		\end{aligned}
	\end{equation}
	Let $\delta>0$ be arbitrary. By Young's inequality we have
	\begin{align*}
		\begin{aligned}
			&\int_\Omega \left|A(x,\nabla u_n)\cdot \nabla \phi_{i,\varepsilon}u_n\right| \,\mathrm{d} x\\
			& \leq\delta\left(\|\nabla u_n\|_p^p+\|\nabla u_n\|_{a,q}^q\right)+
			C_\delta\left(\int_{\Omega}|\nabla \phi_{i,\varepsilon} u_n|^p\,\mathrm{d} x+\int_{\Omega}a(x)|\nabla \phi_{i,\varepsilon} u_n|^q\,\mathrm{d} x\right) \\
			& \leq\delta C_*+
			C_\delta\left(\int_{\Omega}|\nabla \phi_{i,\varepsilon} u_n|^p\,\mathrm{d} x+\int_{\Omega}a(x)|\nabla \phi_{i,\varepsilon} u_n|^q\,\mathrm{d} x\right),
		\end{aligned}
	\end{align*}
	with $C_*>0$ given by the boundedness of $\{u_n\}_{n\in\mathbb{N}}$  and $C_\delta=\delta^{1-p}+\delta^{1-q}$. Thus, \eqref{HH} rewrites as
	\begin{align*}
		&\int_{\Omega}\phi_{i,\varepsilon}|\nabla u_n|^p\,\mathrm{d} x+\int_{\Omega}\phi_{i,\varepsilon}a(x)|\nabla u_n|^q\,\mathrm{d} x\\
		& \leq\theta\int_\Omega\phi_{i,\varepsilon}b_0(x)|u_n|^{p^*}\,\mathrm{d} x+\theta\int_\Omega\phi_{i,\varepsilon}b(x)|u_n|^{q^*}\,\mathrm{d} x \\
		& \quad+\lambda\int_{\Omega}\phi_{i,\varepsilon}w(x)|u_n|^s\,\mathrm{d} x +\delta C_* \\
		& \quad+
		C_\delta\left(\int_{\Omega}|\nabla \phi_{i,\varepsilon} u_n|^p\,\mathrm{d} x+\int_{\Omega}a(x)|\nabla \phi_{i,\varepsilon} u_n|^q\,\mathrm{d} x\right)  +\langle J'(u_n) ,\phi_{i,\varepsilon}u_n \rangle.
	\end{align*}
	Hence, letting $n\to\infty$, by considering \eqref{ps} with $I=J$, \eqref{1} and \eqref{12} as well as Propositions~\ref{prop_CE1} and \ref{prop.CE2}, we get
	\begin{equation}\label{HH2}
		\begin{aligned}
			\int_{\overline{\Omega}}\phi_{i,\varepsilon}\,\mathrm{d}\mu+\int_{\overline{\Omega}}\phi_{i,\varepsilon}\,\mathrm{d}\overline{\mu}
			& \leq\theta\int_{\overline{\Omega}}\phi_{i,\varepsilon} \,\mathrm{d}\nu+\theta\int_{\overline{\Omega}}\phi_{i,\varepsilon}\,\mathrm{d}\overline{\nu} \\
			& \quad+\lambda\int_{\Omega}\phi_{i,\varepsilon} w(x)|u|^s\,\mathrm{d} x +\delta C_*  \\
			& \quad+ C_\delta\left(\int_{\Omega}|\nabla \phi_{i,\varepsilon} u|^p\,\mathrm{d} x+\int_{\Omega}a(x)|\nabla \phi_{i,\varepsilon} u|^q\,\mathrm{d} x\right).
		\end{aligned}
	\end{equation}
	Since by Proposition \ref{prop_CE1} we know that $u\in L^{\mathcal G}(\Omega)$ with
	\begin{align*}
		\mathcal G(x,t):=|t|^{p^*}+a(x)^{\frac{q^*}{q}}|t|^{q^*},
	\end{align*}
	we can apply H\"older's inequality to obtain
	\begin{equation}\label{HH3}
		\begin{aligned}
			&\int_{\Omega_{i,\varepsilon}}|u\nabla \phi_{i,\varepsilon}|^p\,\mathrm{d} x+\int_{\Omega_{i,\varepsilon}}a(x)|u\nabla \phi_{i,\varepsilon}|^q\,\mathrm{d} x\\
			& \leq
			\|u\|_{L^{p^*}(\Omega_{i,\varepsilon})}^p\|\nabla \phi_{i,\varepsilon}\|_{L^N(B(x_i,\varepsilon))}^p
			+\|u\|_{L^{q^*}(a^{q^*/q},\Omega_{i,\varepsilon})}^q\|\nabla \phi_{i,\varepsilon}\|_{L^N(B(x_i,\varepsilon))}^q,
		\end{aligned}
	\end{equation}
	where $\Omega_{i,\varepsilon}=\Omega\cap B(x_i,\varepsilon)$. By a simple change of variable we get
	\begin{equation}\label{HH4}
		\|\nabla \phi_{i,\varepsilon}\|_{L^N(B(x_i,\varepsilon))}=\|\nabla \phi\|_{L^N(B(0,1))}.
	\end{equation}
	Thus, by sending $\varepsilon\to0$ in \eqref{HH3} and considering \eqref{HH4}, it follows that
	\begin{equation}\label{HH5}
		\lim_{\varepsilon\to0}\left(\int_{\Omega}|u\nabla \phi_{i,\varepsilon}|^p\,\mathrm{d} x+\int_{\Omega}a(x)|u\nabla \phi_{i,\varepsilon}|^q\,\mathrm{d} x\right)=0.
	\end{equation}
	On the other hand, by sending $\varepsilon\to0$ in \eqref{HH2}, using \eqref{PL.S.m-n1}, \eqref{PL.S.m-n2} and \eqref{HH5}, we obtain
	\begin{align*}
		\mu_i+\overline{\mu}_i\leq \theta\left(\nu_i+\overline{\nu}_i\right)+ \delta C_*.
	\end{align*}
	Since $\delta>0$ was chosen arbitrarily, the proof of Step 1 is completed.

\vskip5pt
	\noindent{\bf Step 2:} $\nu_i=\overline{\nu}_i=0$ for any $i\in\mathcal I$.
\vskip5pt

	Let us assume by contradiction that there exists $i\in\mathcal I$ such that $\nu_i+\overline{\nu}_i>0$. From \eqref{PL.S.m-n1} and \eqref{PL.S.m-n2}, we have $\mu_i+\overline{\mu}_i>0$ and
	\begin{align}\label{PL.S.E2}
		\nu_i+\overline{\nu}_i\leq S_*\left(\mu_i^{p^*/p}+\overline{\mu}_i^{q^*/q}\right),
	\end{align}
	where $S_*:=\max\left\{S_p^{-p^*},S_p^{-q^*}\right\}$. Now, we claim that there exists $\widetilde{C}_1:=\widetilde{C}_1(N,p,q)>0$ such that
	\begin{align}\label{PL.S.E3}
		\theta\left(\nu_i+\overline{\nu}_i\right)\geq \mu_i+\overline{\mu}_i\geq \widetilde{C}_1\min\left\{\theta^{-\frac{p}{p^*-p}},\theta^{-\frac{q}{q^*-q}}\right\}.
	\end{align}
	For this, we distinguish the following cases:
	\begin{enumerate}[leftmargin=1.5cm]
		\item[\bf Case 1:]
			Let $\mu_i\geq 1$.\\
			Then, combining Step 1 and \eqref{PL.S.E2}, we get
			\begin{align*}
				\theta^{-1}\left(\mu_i+\overline{\mu}_i\right)\leq 	\nu_i+\overline{\nu}_i\leq S_*\left(\mu_i^{q^*/q}+\overline{\mu}_i^{q^*/q}\right)\leq S_*\left(\mu_i+\overline{\mu}_i\right)^{q^*/q}.
			\end{align*}
			This yields
			\begin{align*}
				\theta\left(\nu_i+\overline{\nu}_i\right)\geq \mu_i+\overline{\mu}_i\geq  S_*^{-\frac{q}{q^*-q}}\theta^{-\frac{q}{q^\ast-q}}.
			\end{align*}
		\item[\bf Case 2:]
			Let $\mu_i<1$ and $\overline{\mu}_i\geq 1$.\\
			Then, from Step 1 and \eqref{PL.S.E2} again, we obtain
			\begin{align*}
				\theta^{-1}\left(\mu_i+\overline{\mu}_i\right)\leq \nu_i+\overline{\nu}_i\leq 2S_*\overline{\mu}_i^{q^*/q}\leq 2S_*\left(\mu_i+\overline{\mu}_i\right)^{q^*/q}.
			\end{align*}
			This gives
			\begin{align*}
				\theta\left(\nu_i+\overline{\nu}_i\right)\geq \mu_i+\overline{\mu}_i\geq  \left(2S_*\right)^{-\frac{q}{q^*-q}}\theta^{-\frac{q}{q^\ast-q}}.
			\end{align*}
		\item[\bf Case 3:]
			Let $\mu_i< 1$ and $\overline{\mu}_i<1$.\\
			Again from Step 1 and \eqref{PL.S.E2}, it follows that
			\begin{align*}
				\theta^{-1}\left(\mu_i+\overline{\mu}_i\right)\leq\nu_i+\overline{\nu}_i\leq S_*\left(\mu_i^{p^*/p}+\overline{\mu}_i^{p^*/p}\right)\leq S_*\left(\mu_i+\overline{\mu}_i\right)^{p^*/p}.
			\end{align*}
			This and the fact that $\mu_i+\overline{\mu}_i>0$ lead to
			\begin{align*}
				\theta\left(\nu_i+\overline{\nu}_i\right)\geq \mu_i+\overline{\mu}_i \geq  S_*^{-\frac{p}{p^*-p}}\theta^{-\frac{p}{p^\ast-p}}.
			\end{align*}
	\end{enumerate}
	In summary, we arrive at the statement \eqref{PL.S.E3} by taking
	\begin{align*}
		\widetilde{C}_1:=\min\left\{\left(2S_*\right)^{-\frac{q}{q^*-q}},S_*^{-\frac{p}{p^*-p}}\right\}.
	\end{align*}
	On the other hand, putting $\widetilde{q}=(q+p^*)/2$, by \eqref{ps} we have
	\begin{align*}
		c&= J(u_n)-\frac{1}{\widetilde{q}}\langle J'(u_n) ,u_n\rangle+o_n(1) \\
		&\geq  \left(\frac{1}{q}-\frac{1}{\widetilde{q}}\right)\|\nabla u_n\|_{a,q}^q-\lambda \left(\frac{1}{s}-\frac{1}{\widetilde{q}}\right)\int_{\Omega}w(x)|u_n|^{s}\,\mathrm{d} x\\
		&\quad +\theta\left(\frac{1}{\widetilde{q}}-\frac{1}{p^\ast}\right) \left[\int_{\Omega}b_0(x)|u_n|^{p^\ast}\,\mathrm{d} x+\int_{\Omega}b(x)|u_n|^{q^\ast}\,\mathrm{d} x\right]+o_n(1),
	\end{align*}
	as $n\to\infty$. Passing to the limit as $n\to\infty$ in the last estimate and using \eqref{1}--\eqref{PL.S.m-n2}  
	and Proposition~\ref{prop.CE2}, we obtain
	\begin{equation}\label{3.16}
		\begin{aligned}
			c &\geq  \left(\frac{1}{q}-\frac{1}{\widetilde{q}}\right)\|\nabla u\|_{a,q}^q-\lambda \left(\frac{1}{s}-\frac{1}{\widetilde{q}}\right)\int_{\Omega}w(x)|u|^{s}\,\mathrm{d} x \\
			&\quad +\theta\left(\frac{1}{\widetilde{q}}-\frac{1}{p^\ast}\right) \left(\|u\|_{b_0,p^\ast}^{p^\ast}+\|u\|_{b,q^\ast}^{q^\ast}+\nu_i+\overline{\nu}_i\right).
		\end{aligned}
	\end{equation}
	By \eqref{H3} we have
	\begin{align*}
		\int_{\{b=0\}}w(x)|u|^{s}\,\mathrm{d} x\leq \int_{\{b>0\}}|w(x)||u|^{s}\,\mathrm{d} x+C_w\left(\int_{\Omega}a(x)|\nabla u|^q\,\mathrm{d} x\right)^{\frac{s}{q}}.
	\end{align*}
	Thus, using H\"older's inequality and Young's inequality gives
	\begin{align*}
		& \lambda \left(\frac{1}{s}-\frac{1}{\widetilde{q}}\right)	\int_{\Omega}w(x)|u|^{s}\,\mathrm{d} x  \\
		 & \leq 2\lambda \left(\frac{1}{s}-\frac{1}{\widetilde{q}}\right)\int_{\Omega}|w(x)|\chi_{\{b>0\}}|u|^{s}\,\mathrm{d} x+C_w\lambda \left(\frac{1}{s}-\frac{1}{\widetilde{q}}\right)\|\nabla u\|_{a,q}^{s}  \\
		 & \leq2\left(\frac{1}{s}-\frac{1}{\widetilde{q}}\right)\|w\chi_{\{b>0\}}b^{-\frac{s}{q^\ast}}\|_{L^{\frac{q^\ast}{q^\ast-s}}(\Omega)}\lambda\|u\|_{b,q^\ast}^{s}\\
		 &\quad +\frac{1}{2}\left(\frac{1}{q}-\frac{1}{\widetilde{q}}\right)\|\nabla u\|_{a,q}^q+C_2\lambda^{\frac{q}{q-s}},
	\end{align*}
	with a suitable $C_2:=C_2(N,p,q,s,w)>0$. From \eqref{PL.S.E3}, \eqref{3.16} and the last estimate we deduce
	\begin{align*}
		c\geq & -2\left(\frac{1}{s}-\frac{1}{\widetilde{q}}\right)\|w\chi_{\{b>0\}}b^{-\frac{s}{q^\ast}}\|_{L^{\frac{q^\ast}{q^\ast-s}}(\Omega)}\lambda\|u\|_{b,q^\ast}^{s}+\theta\left(\frac{1}{\widetilde{q}}-\frac{1}{p^\ast}\right)\|u\|_{b,q^\ast}^{q^\ast} \\
		& + C_1\min\left\{\theta^{-\frac{p}{p^*-p}},\theta^{-\frac{q}{q^*-q}}\right\}-C_2\lambda^{\frac{q}{q-s}},
	\end{align*}
	where $C_1:=\left(\displaystyle\frac{1}{\widetilde{q}}-\frac{1}{p^\ast}\right)\widetilde{C}_1$. That is, we get
	\begin{align} \label{PL.S.E4}
		c\geq h_{\lambda,\theta}\left(\|u\|_{b,q^\ast}\right)+C_1\min\left\{\theta^{-\frac{p}{p^*-p}},\theta^{-\frac{q}{q^*-q}}\right\}-C_2\lambda^{\frac{q}{q-s}},
	\end{align}
	where
	\begin{align*}
		h_{\lambda,\theta}(t):=\theta{d_1} t^{q^\ast}-\lambda{d_2} t^s\quad\text{for } t\geq 0,
	\end{align*}
	with
	\begin{align*}
		{d_1}:=\left(\frac{1}{\widetilde{q}}-\frac{1}{p^\ast}\right)\quad\text{and}\quad{d_2}:=2\left(\frac{1}{s}-\frac{1}{\widetilde{q}}\right)\|w\chi_{\{b>0\}}b^{-\frac{s}{q^\ast}}\|_{L^{\frac{q^\ast}{q^\ast-s}}(\Omega)}.
	\end{align*}
	Since
	\begin{align*}
		\min_{t\geq 0}\, h_{\lambda,\theta}(t)=h_{\lambda,\theta}\left( \left(\frac{\lambda{d_2} s}{\theta{d_1} q^\ast}\right)^{\frac{1}{q^\ast-s}} \right)=-\frac{q^*-s}{q^*}\left(\frac{s}{q^\ast}\right)^{\frac{s}{q^\ast-s}}(\lambda{d_2})^{\frac{q^\ast}{q^\ast-s}}(\theta{d_1})^{-\frac{s}{q^\ast-s}},
	\end{align*}
	we derive from \eqref{PL.S.E4} that
	\begin{align*}
		c\geq C_1\min\left\{\theta^{-\frac{p}{p^*-p}},\theta^{-\frac{q}{q^*-q}}\right\}-C_2\lambda^{\frac{q}{q-s}}-C_3\lambda^{\frac{q^\ast}{q^\ast-s}}\theta^{-\frac{s}{q^\ast-s}},
	\end{align*}
	with a suitable $C_3:=C_3\left(N,p,q,s,w,b\right)>0$, which contradicts to \eqref{PS_c2}. This proves Step 2.

\vskip5pt
	\noindent{\bf Step 3:}  $u_n\to u$ in $W_0^{1,\mathcal{H}}(\Omega)$ as $n\to\infty$.
\vskip5pt

	By Step 2, $\nu_i=\overline{\nu}_i=0$ for any $i\in\mathcal{I}$ and so by \eqref{PL.S.m-n1} and \eqref{PL.S.m-n2}, we obtain
	\begin{align*}
		\int_{\Omega}b_0(x)|u_n|^{p^\ast}\,\mathrm{d} x\to \int_{\Omega}b_0(x)|u|^{p^\ast}\,\mathrm{d} x\quad\text{and}\quad \int_{\Omega}b(x)|u_n|^{q^\ast}\,\mathrm{d} x\to \int_{\Omega}b(x)|u|^{q^\ast}\,\mathrm{d} x.
	\end{align*}
	From this and \eqref{1} we conclude that
	\begin{equation}
		\label{BL}
		\int_{\Omega}b_0(x)|u_n-u|^{p^\ast}\,\mathrm{d} x\to 0\quad\text{and}\quad\int_{\Omega}b(x)|u_n-u|^{q^\ast}\,\mathrm{d} x\to 0
	\end{equation}
	in view of the Br\'ezis-Lieb lemma (see e.g. Ho--Sim \cite[Lemma 3.6]{Ho-Sim-2016}).
	From \eqref{ps} we have $\langle J'(u_n),u_n-u\rangle\to0$ as $n\to\infty$, which yields
	\begin{align*}
		o_n(1)&= \int_\Omega A(x,\nabla u_n)\cdot\nabla(u_n-u)\,\mathrm{d} x-\lambda \int_{\Omega}w(x)|u_n|^{s-2}u_n(u_n-u)\,\mathrm{d} x \\
		& \quad+\theta\int_{\Omega}B(x,u_n)(u_n-u)\,\mathrm{d} x.
	\end{align*}
	From this, by H\"older's inequality combined with Proposition \ref{prop.CE2} and \eqref{BL}, we get
	\begin{equation}
		\label{aaa}
		\lim_{n\to\infty}\int_\Omega A(x,\nabla u_n)\cdot\nabla(u_n-u)\,\mathrm{d} x=0.
	\end{equation}
	Then, combining \eqref{1}, \eqref{aaa} and Proposition \ref{P2.4}, we deduce that $u_n\to u$ in $W_0^{1,\mathcal{H}}(\Omega)$. The proof  is complete.
\end{proof}

\section{Proof of Theorem \ref{Theo.NS}}\label{Section_4}

For the proof of Theorem \ref{Theo.NS} we employ the idea by Ho--Sim \cite{Ho-Sim-2021}. However, we minimize the functional $J_+$ on a suitable ball
\begin{align*}
	B_r=\left\{u\in W_0^{1,\mathcal{H}}(\Omega)\colon\,\, \|u\|< r\right\},
\end{align*}
working with the Luxemburg norm $\|\cdot\|$.

Before we prove Theorem \ref{Theo.NS}, we first need the following lemma.

\begin{lemma}\label{negativo}
	Let hypotheses \eqref{H1}--\eqref{H3} be satisfied and let $\lambda>\lambda_0$ with $\lambda_0$ as defined in \eqref{la*}. Then, there exists $\widetilde{\theta}_\ast=\widetilde{\theta}_\ast(\lambda)>0$ such that for any $\theta\in \left(0,\widetilde{\theta}_\ast\right)$, there exist $r$, $\beta>0$ such that
	\begin{align*}
		\underset{u\in \partial B_r}{\inf} J_+(u)\geq \beta>0> \underset{u\in B_r}{\inf} J_+(u).
	\end{align*}
\end{lemma}

\begin{proof}
	Fix $\lambda>\lambda_0$. Thus, there exists $\psi\in \mathcal{W}_+$ such that
	\begin{align}\label{PT1.la}
		C_0\left(\frac{\int_\Omega a(x)|\nabla \psi|^q\,\mathrm{d} x}{q}\right)^{\frac{s-p}{q-p}}\left(\frac{\int_\Omega |\nabla \psi|^p\,\mathrm{d} x}{p}\right)^{\frac{q-s}{q-p}}\frac{s}{\int_\Omega w(x)\psi_+^s\,\mathrm{d} x}<\lambda.
	\end{align}
	Now, for $t>0$, we have
	\begin{align*}
		J_+(t\psi)=t^p g_\lambda(t)-\theta\xi(t),
	\end{align*}
	where
	\begin{align*}
		g_\lambda(t):=\alpha_1-\alpha_2\lambda t^{s-p}+\alpha_3t^{q-p}
	\end{align*}
	with
	\begin{align*}
		\alpha_1:=\frac{1}{p}\|\nabla \psi\|_p^p>0,\quad \alpha_2:=\frac{1}{s}\int_\Omega w(x)\psi_+^{s}\,\mathrm{d} x>0\quad \text{and}\quad \alpha_3:=	\frac{1}{q}\|\nabla \psi\|_{a,q}^q>0,
	\end{align*}
	and
	\begin{align*}
		\xi(t):= \frac{1}{p^\ast}\left(\int_\Omega b_0(x)\psi_+^{p^\ast}\,\mathrm{d} x\right)t^{p^\ast}+\frac{1}{q^\ast}\left(\int_\Omega b(x)\psi_+^{q^\ast}\,\mathrm{d} x\right)t^{q^\ast}.
	\end{align*}
	Due to \eqref{PT1.la} and considering \eqref{C0}, we have
	\begin{align*}
		\min_{t>0}\, g_\lambda(t)
		&=g_\lambda\left(\left(\frac{s-p}{q-p}\alpha_2\alpha_3^{-1}\lambda\right)^{\frac{1}{q-s}}\right)\\
		&=\alpha_1-\frac{q-s}{q-p}\left(\frac{s-p}{q-p}\right)^{\frac{s-p}{q-s}}\alpha_3^{-\frac{s-p}{q-s}}\alpha_2^{\frac{q-p}{q-s}}\lambda^{\frac{q-p}{q-s}}<0.
	\end{align*}
	Thus, by setting $t_0=t_0(\lambda)>0$ as
	\begin{align*}
		t_0:=\left(\frac{s-p}{q-p}\alpha_2\alpha_3^{-1}\lambda\right)^{\frac{1}{q-s}},
	\end{align*}
	we conclude that $g_\lambda(t_0)<0$, and so
	\begin{align}\label{PT1.t0}
		J_+(t_0\psi)=[t_0]^pg_\lambda(t_0)-\theta\xi(t_0)<0.
	\end{align}
	We take $r=r(\lambda)>0$ as
	\begin{align}
		\label{PT1.r}
		r:=\max\left\{1+t_0\|\psi\|,\left(\frac{2qC_w\lambda}{s}\right)^{\frac{pq}{q-s}}\right\},
	\end{align}
	where $C_w$ is given in \eqref{H3}. Then, for $\|u\|=r  >1$, by \eqref{H3}, Propositions \ref{P.m-n} and \eqref{Ce}, it follows that
	\begin{equation}\label{PT1.E1}
		\begin{aligned}
			J_+(u) & \geq \frac{1}{p}\|\nabla u\|_p^p+\frac{1}{q}\|\nabla u\|_{a,q}^q-\frac{C_w\lambda}{s}\|\nabla u\|_{a,q}^s-\frac{\theta}{p^\ast}\max\left\{\|u\|_{\mathcal{B}}^{p^\ast},\|u\|_{\mathcal{B}}^{q^\ast}\right\} \\
			& \geq \frac{1}{p}\|\nabla u\|_p^p+\frac{1}{q}\|\nabla u\|_{a,q}^q-\frac{C_w\lambda}{s}\|\nabla u\|_{a,q}^s-\frac{\theta C_e^{q^*}}{p^\ast}\|u\|^{q^*}.
		\end{aligned}
	\end{equation}
	We claim that there exist $\widetilde{\theta}_\ast=\widetilde{\theta}_\ast(\lambda)>0$ such that, for any $\theta\in \left(0,\widetilde{\theta}_\ast\right)$, we have
	\begin{equation}
		\label{PT1.Geo.partial B_r}
		J_+(u)\geq\beta,\quad\text{for any } u\in\partial B_r,
	\end{equation}
	with a suitable $\beta=\beta(\lambda,\theta)>0$.
	For this, we distinguish the following two cases:
	\begin{enumerate}[leftmargin=1.5cm]
		\item[\bf Case 1:]
			Let $\displaystyle\frac{1}{q}\|\nabla u\|_{a,q}^q\leq \displaystyle\frac{2C_w\lambda}{s}\|\nabla u\|^s_{a,q}$, i.e., $\|\nabla u\|_{a,q}\leq \left(\displaystyle\frac{2qC_w\lambda}{s}\right)^{\displaystyle\frac{1}{q-s}}$.\\
			From \eqref{PT1.r} we have
			\begin{align*}
				\frac{C_w\lambda}{s}\|\nabla u\|^s_{a,q}
				\leq \frac{C_w\lambda}{s}\left(\frac{2qC_w\lambda}{s}\right)^{\frac{s}{q-s}}\leq \frac{1}{2q}r^p=\frac{1}{2q}\|u\|^p.
			\end{align*}
			Using these estimates, we derive from Proposition \ref{P.m-n} and \eqref{PT1.E1} that
			\begin{align*}
				J_+(u)&\geq \frac{1}{q}\|u\|^p-\frac{1}{2q}\|u\|^p-\frac{\theta C_e^{q^*}}{p^\ast}\|u\|^{q^*}\\&=\frac{C_e^{q^*}r^{q^*}}{p^\ast}\left(\frac{p^*}{2qC_e^{q^*}}r^{p-q^*}-\theta\right).
			\end{align*}
		\item[\bf Case 2:]
			Let $\displaystyle\frac{1}{q}\|\nabla u\|_{a,q}^q\geq \displaystyle\frac{2C_w\lambda}{s}\|\nabla u\|^s_{a,q}$.\\
			In this case, we easily derive from Proposition \ref{P.m-n} and \eqref{PT1.E1} that
			\begin{align*}
				J_+(u) & \geq \frac{1}{p}\|\nabla u\|_{p}^p+\frac{1}{2q}\|\nabla u\|_{a,q}^q-\frac{\theta C_e^{q^*}}{p^\ast}\|u\|^{q^*}\\&\geq \frac{1}{2q}\|u\|^p-\frac{\theta C_e^{q^*}}{p^\ast}\|u\|^{q^*} \\
				& =\frac{C_e^{q^*}r^{q^*}}{p^\ast}\left( \frac{p^*}{2qC_e^{q^*}}r^{p-q^*}-\theta\right).
			\end{align*}
	\end{enumerate}
	In any case, by taking
	\begin{align*}
		\widetilde{\theta}_\ast:=\displaystyle\frac{p^*}{2qC_e^{q^*}}r^{p-q^*}\quad\text{and}\quad\beta:=\frac{C_e^{q^*}r^{q^*}}{p^\ast}(\widetilde{\theta}_\ast-\theta),
	\end{align*}
	the statement \eqref{PT1.Geo.partial B_r} holds true for any $\theta\in \left(0,\widetilde{\theta}_\ast\right)$.

	Finally, note that $t_0\psi\in B_r$ by \eqref{PT1.r}. Hence, \eqref{PT1.t0} yields
	\begin{align*}
		\underset{u\in B_r}{\inf} J_+(u)\leq J_+(t_0\psi)<0.
	\end{align*}
	This and \eqref{PT1.Geo.partial B_r} complete the proof.
\end{proof}

\begin{proof}[Proof of Theorem \ref{Theo.NS}]
	Note that for $\theta\in(0,1)$, the right-hand side of \eqref{PS_c2} can be rewritten as
	\begin{align*}
		\overline{c}_{\lambda,\theta}:=C_1\theta^{-\frac{q}{q^*-q}}-C_3\lambda^{\frac{q^\ast}{q^\ast-s}}\theta^{-\frac{s}{q^\ast-s}}-C_2\lambda^{\frac{q}{q-s}}.
	\end{align*}
	For the case $b\equiv 0$, we take $\overline{c}_{\lambda,\theta}:=C_1\theta^{-\frac{p}{p^*-p}}-C_2\lambda^{\frac{q}{q-s}}$, which is the right-hand side of \eqref{PS_c2'}. For a fixed $\lambda>0$, since $s<q$, there exists $\widehat{\theta}_*=\widehat{\theta}_*(\lambda)>0$ sufficiently small such that $\overline{c}_{\lambda,\theta}>0$ for any $\theta\in\left(0,\widehat{\theta}_*\right)$. Thus, let us fix $\lambda>\lambda_0$, with $\lambda_0$ as defined in \eqref{la*}. Next, let us fix $\theta\in\left(0,\theta_*\right)$ with $\theta_*:=\min\left\{\widehat{\theta}_*,\widetilde{\theta}_*,1\right\}$, where $\widetilde{\theta}_*$ is as in Lemma~\ref{negativo}.
	Thanks to Lemma \ref{negativo}, we can apply Ekeland's variational principle to $J_+$ which provides a minimizing sequence $\{u_n\}_{n\in\mathbb{N}}\subset B_r$ such that
	\begin{align*}
		J_+(u_{n})\rightarrow m_r \quad\text{and}\quad J'_+(u_{n})\rightarrow 0,
	\end{align*}
	where
	\begin{align*}
		m_r:=\inf_{u\in B_r} J_+(u).
	\end{align*}
	Furthermore, since $m_r<0<\overline{c}_{\lambda,\theta}$ due to Lemma \ref{negativo}, we can apply Lemma \ref{Le.PS} for the sequence $\{u_n\}_{n\in\mathbb{N}}$, so that there exists $u\in W_0^{1,\mathcal{H}}(\Omega)$ such that $u_n\to u$ in $W_0^{1,\mathcal{H}}(\Omega)$. Hence
	\begin{align}\label{PT1.2}
		J_+'(u)=0\quad \text{and}\quad J_+(u)=m_r<0.
	\end{align}
Thus, we have
	\begin{align*}
	0&= \langle J_+'(u), u_-\rangle =\int_\Omega\mathcal A\left(x,\nabla u_-\right)\,\mathrm{d} x=0,
\end{align*}
where $u_-:=\max\{-u,0\}$. It follows that $u_-=0$ a.e.\,in $\Omega$ and so $u\geq 0$ a.e.\,in $\Omega$. From this and \eqref{PT1.2} we obtain
\begin{align*}
	J'(u)=0\quad \text{and}\quad J(u)=m_r<0.
\end{align*}
Therefore, $u$ is a nontrivial nonnegative solution to problem~\eqref{Eq.Sandwich}. This proves the assertion of the theorem.
\end{proof}

\section{Proof of Theorem \ref{Theo.MS}}\label{Section_5}

The proof of Theorem \ref{Theo.MS} is inspired by Ho--Sim \cite[Theorem 1.3]{Ho-Sim-2023}. However, unlike the case of the $(p,q)$-Laplacian, $\|\nabla \cdot\|_{a,q}$ is no longer an equivalent norm on $W_0^{1,\mathcal{H}}(\Omega)$, and this makes the situation for the double phase operator much more complicated. Furthermore, we require the existence of a ball $B\subset\Omega_+$ as set in Theorem \ref{Theo.MS}, in order to get the technical result of Lemma \ref{genus.k}.

Let $\theta>0$ and $\lambda>0$. We easily observe that the functional $J$ is not bounded from below in $W_0^{1,\mathcal{H}}(\Omega)$ because of the presence of the critical Sobolev term. For this, we mainly work with a truncated functional. Let $1<t_1<t_2$ and choose a cut-off function $\phi\in C_c^\infty(\R)$ being non-increasing with $0\leq \phi(\cdot)\leq 1$ such that
\begin{align}\label{cutoff}
	\phi(t)=1\quad\text{for } |t|\leq t_1\quad\text{and}\quad \phi(t)=0\quad\text{for }|t|\geq t_2.
\end{align}
Now we can define the truncated functional ${J_\phi}\colon W_0^{1,\mathcal{H}}(\Omega) \to\R$ as
\begin{align}\label{jphi}
	{J_\phi}(u):=\int_\Omega \mathcal{A}(x,\nabla u) \,\mathrm{d} x-\frac{\lambda}{s}\int_\Omega w(x)|u|^s\,\mathrm{d} x
	+\phi(\|u\|)\,\theta\int_\Omega \widehat{B}(x,u) \,\mathrm{d} x.
\end{align}
Obviously, ${J_\phi}\in C^1(W_0^{1,\mathcal{H}}(\Omega),\R)$ by the definition of $\phi$ and by Colasuonno--Squassina \cite[Proposition 3.2]{Colasuonno-Squassina-2016}.

In order to prove the existence of multiple solutions for \eqref{Eq.Sandwich}, we need the Krasnoselskii's genus theory. For this, we first recall the definition of the genus and denote
\begin{align*}
	\Sigma=\left\{A\subset W_0^{1,\mathcal{H}}(\Omega)\setminus\{0\}\colon\,\, A \text{  is closed and symmetric} \right\}.
\end{align*}
The genus  $\gamma(A)$ of $A\in\Sigma$	is defined as the smallest positive integer $d$ such that there exists an odd continuous map from $A$ to $\R^{d}\setminus\{0\}$.  If such $d$ does not exist, then we set $\gamma(A)=\infty$. Also, we define $\gamma(\emptyset)=0$. We refer to Rabinowitz \cite{Rabinowitz-1986} for more details on this topic.

We get a suitable property for sublevels of functional ${J_\phi}$.

\begin{lemma}\label{genus.k}
	Let hypotheses \eqref{H1}--\eqref{H3} be satisfied. Then, for any $j\in\N$, there exists $d_j>0$ such that for any $\lambda>d_j$ and any $\theta>0$, there exists $\varepsilon_j>0$ such that ${J_\phi}^{-\varepsilon_j}\in\Sigma$ and
	\begin{align*}
		\gamma({J_\phi}^{-\varepsilon_j}) \geq j,
	\end{align*}
	with ${J_\phi}^{-\varepsilon_j}:=\{u \in W_0^{1,\mathcal{H}}(\Omega)\colon\,\, J_\phi(u) \leq -\varepsilon_j\}.$
\end{lemma}

\begin{proof}
	Fix $\theta>0$ and let $B\subset\Omega_+$ be as in Theorem \ref{Theo.MS}. For any $j\in\mathbb N$, we define
	\begin{align*}
		X_j:=\operatorname{span}\left\{\varphi_1,\varphi_2,\ldots,\varphi_j\right\},
	\end{align*}
	where $\varphi_j$ is an eigenfunction corresponding to the $j$-th eigenvalue of the following problem
	\begin{align*}
		-\Delta u=\mu u  \quad\text{in } B,\quad u=0\quad \text{in } \partial B,
	\end{align*}
	which can be extended to $\Omega$ by putting $\varphi_j(x)=0$ for $x\in\Omega\setminus B$. Since all norms on $X_j$ are mutually equivalent, noticing  that $\operatorname{supp}\,(u)\subset B$ for $u\in X_j$ and $w(x)>0$ for a.a.\,$x\in B$, we can find $\delta_j>1$ such that
	\begin{align}\label{equi.norms}
		\delta_j^{-1}\max\left\{\|\nabla u\|_{L^{p}(B)},\|\nabla u\|_{L^{q}(a,B)}\right\}\leq \|u\|\leq \delta_j\|u\|_{L^{s}(w,B)},
	\end{align}
	for any $u\in X_j$. Without any loss of generality, we can choose $\{\delta_j\}_{j\in\mathbb{N}}$ such that $\delta_j<\delta_{j+1}$ for any $j\in\N$. We have
	\begin{align*}{J_\phi}(u)\leq \frac{1}{p}\|\nabla u\|_{L^{p}(B)}^{p}+\frac{1}{q}\|\nabla u\|_{L^{q}(a,B)}^{q}-\frac{\lambda }{s}\|u\|_{L^s(w,B)}^{s},\quad\text{for any } u\in X_j.
	\end{align*}
	From this and \eqref{equi.norms}, for any ${\tau}>0$, we infer that
	\begin{align}\label{T}
		{J_\phi}(u)\leq \frac{1}{p}(\delta_j{\tau})^{p}+\frac{1}{q}(\delta_j{\tau})^{q}-\frac{\lambda }{s}(\delta_j^{-1}{\tau})^{s}={\tau}^{p}\,h_\lambda({\tau}),
	\end{align}
	for any $u\in \partial B_{\tau}\cap X_j$, where
	\begin{align*}
		h_\lambda({\tau}):=\alpha_j{\tau}^{q-p}+\beta_j-\gamma_j\lambda{\tau}^{s-p}
	\end{align*}
	with
	\begin{align*}
		\alpha_j:=q^{-1}\delta_j^{q},\quad \beta_j:=p^{-1}\delta_j^{p},\quad\gamma_j:=s^{-1}\delta_j^{-s}.
	\end{align*}
	Let us set $T^*_j=T^*_j(\lambda)>0$ as
	\begin{align*}
		T^*_j:=\left[\frac{(s-p)\gamma_j\,\lambda}{(q-p)\alpha_j}\right]^{\frac{1}{q-s}}.
	\end{align*}
	Then, for
	\begin{align}\label{PT2.dk}
		d_j:=\left(\frac{q-p}{q-s}\right)^{\frac{q-s}{q-p}}\left(\frac{q-p}{s-p}\right)^{\frac{s-p}{q-p}}\alpha_j^{\frac{s-p}{q-p}}\beta_j^{\frac{q-s}{q-p}}\gamma_j^{-1}=C_4\delta_j^{2s}
	\end{align}
	with $C_4=C_4(p,q,s)$, it holds that
	\begin{equation}
		\label{bar-h(T*)}
		h_\lambda(T^*_j)=\beta_j-\frac{q-s}{q-p}\left(\frac{s-p}{q-p}\right)^{\frac{s-p}{q-s}}\alpha_j^{-\frac{s-p}{q-s}}\gamma_j^{\frac{q-p}{q-s}}\lambda^{\frac{q-p}{q-s}}<0,\quad\text{for any } \lambda>d_j.
	\end{equation}
	Thus, for any $\lambda>d_j$, we get
	\begin{align*}{J_\phi}(u)\leq (T^*_j)^{p}h_\lambda(T^*_j)=:-\varepsilon_j<0,\quad\text{for any } u\in \partial B_{T^*_j}\cap X_j
	\end{align*}
	in view of \eqref{T} and \eqref{bar-h(T*)}. Hence, $\partial B_{T^*_j}\cap X_j\subset {J_\phi}^{-\varepsilon_j}$. Clearly, $\partial B_{T^*_j}\cap X_j\in\Sigma$ and ${J_\phi}^{-\varepsilon_j}\in\Sigma$. Therefore, by standard properties of the genus as in Rabinowitz \cite[Proposition 7.7]{Rabinowitz-1986}, we obtain
	\begin{align*}
		\gamma( {J_\phi}^{-\varepsilon_j})\geq \gamma(\partial B_{T^*_j}\cap X_j)=j.
	\end{align*}
	The proof is complete.
\end{proof}

Now, we are going to construct an appropriate minimax sequence of negative critical values for the truncated functional $J_\phi$.	For any $j\in\N$, define the minimax values $c_j=c_j(\lambda,\theta)$ as
\begin{equation}
	\label{cn}
	c_j:= \inf_{A\in \Sigma_j} \sup_{u \in
		A}{J_\phi}(u),\quad\text{where }\Sigma_j:=\{A\in\Sigma\colon\,\, \gamma(A) \geq j\}.
\end{equation}
This definition is well defined since $\partial(X_j\cap B_{\tau})\in\Sigma_j$ for any ${\tau}>0$. Clearly, for any $j\in\N$, it holds that $c_{j+1}\leq c_j$.

We have the following properties for $\{c_j\}_{j\in\mathbb{N}}$.

\begin{lemma}\label{Le.c_n<0}
	Let hypotheses \eqref{H1}--\eqref{H3} be satisfied and let $\theta>0$, $\lambda>d_j$ with $d_j$ as given in \eqref{PT2.dk}, and let $\{c_j\}_{j\in\N}$ be defined as in \eqref{cn}. Then, for any $j\in\N$  we have that
	\begin{align*}
		-\infty<c_j<0.
	\end{align*}
\end{lemma}

\begin{proof}
	Fix $\theta>0$ and $\lambda>d_j$. By Lemma \ref{genus.k} there exists $\varepsilon_j>0$ such that ${J_\phi}^{-\varepsilon_j}\in \Sigma_j$ and so
	\begin{align*}
		c_j= \inf_{A\in
			\Sigma_j} \sup_{u \in A}\, {J_\phi}(u) \leq\sup_{u\in {J_\phi}^{-\varepsilon_j}}\, {J_\phi}(u) \leq
		-\varepsilon_j <0.
	\end{align*}
	On the other hand, by \eqref{jphi} with $\phi\in C_c^\infty(\R)$ satisfying \eqref{cutoff}, and the fact that $s<q$, we easily see that ${J_\phi}$ is bounded from below which yields
	\begin{align*}
		c_j>-\infty.
	\end{align*}
	This completes the proof.
\end{proof}

In order to get solutions of \eqref{Eq.Sandwich}, we need to go back to the main functional $J$. Thus, we properly choose $t_1$ and $t_2$ in \eqref{cutoff}. For this, using \eqref{H3} and \eqref{Ce}, we have
\begin{align}\label{PT2.E1}
	J(u)\geq\frac{1}{p}\|\nabla u\|_p^{p}+\frac{1}{q}\|\nabla u\|_{a,q}^{q}-\lambda\widetilde{k}_1\|\nabla u\|_{a,q}^{s}-\theta k_2\|u\|^{q^*}
\end{align}
for any $u\in W_0^{1,\mathcal{H}}(\Omega)$ with $\|u\|\geq 1$, where $\widetilde{k}_1:=C_w/s$ and $k_2:=C_e^{q^*}/p^\ast$.

Fix $m\in (1,p)$ and let $\delta\in (0,1)$ be such that
\begin{align}\label{delta}
	s=\delta m+(1-\delta)q\quad\text{i.e.,}\quad \delta=\frac{q-s}{q-m}.
\end{align}
By Young's inequality we have
\begin{equation}\label{PT2.E2}
	\begin{aligned}
		\lambda\widetilde{k}_1\|\nabla u\|_{a,q}^{s} & =(2q(1-\delta))^{\delta-1}\|\nabla u\|_{a,q}^{(1-\delta)q}(2q(1-\delta))^{1-\delta}\lambda\widetilde{k}_1\|\nabla u\|_{a,q}^{\delta m}  \\
		& \leq \frac{1}{2q}\|\nabla u\|_{a,q}^{q}+\delta(2q(1-\delta))^{\frac{1-\delta}{\delta}}(\lambda\widetilde{k}_1)^{\frac{1}{\delta}}\|\nabla u\|_{a,q}^{m}.
	\end{aligned}
\end{equation}
Combining Proposition \ref{P.m-n}, \eqref{PT2.E1} with \eqref{PT2.E2}, we get
\begin{align*}
	J(u)\geq\frac{1}{2q}\|u\|^{p}-\lambda^{\frac{1}{\delta}}k_1\|u\|^{m}-\theta k_2\|u\|^{q^*}\quad\text{for any } u\in W_0^{1,\mathcal{H}}(\Omega) \text{ with } \|u\|\geq 1,
\end{align*}
where
\begin{align}\label{k1}
	k_1:=\delta(2q(1-\delta))^{\frac{1-\delta}{\delta}}(\widetilde{k}_1)^{\frac{1}{\delta}},
\end{align}
that is,
\begin{align}\label{PTcc.gi}
	J(u)\geq  f_{\lambda,\theta}(\|u\|)\,\, \text{ for any } \, u\in W_0^{1,\mathcal{H}}(\Omega)\, \text{ with }\, \|u\|\geq1,
\end{align}
where
\begin{align*}
	f_{\lambda,\theta}(t):=\frac{1}{2q}t^{p}-\lambda^{\frac{1}{\delta}} k_1t^{m}-\theta k_2 t^{q^*},\quad t\geq0.
\end{align*}

Now, we will check the location of critical points of $f_{\lambda,\theta}$.
For this, we set
\begin{align*}
	f_{\lambda,\theta}(t)=k_1t^{m}\,\widetilde{f}_{\lambda,\theta}(t),\quad \text{with }\,\,\widetilde{f}_{\lambda,\theta}(t):=-\lambda^{\frac{1}{\delta}} +a_0t^{p-m}-\theta b_0t^{q^*-m},
\end{align*}
where
\begin{equation}
	\label{a0}
	a_0:=(2q k_1)^{-1}>0,\quad b_0:=k_2k_1^{-1}>0.
\end{equation}
Clearly, $\widetilde{f}\,'_{\lambda,\theta}(t)>0$ for $t\in (0,T_*)$ and $\widetilde{f}\,'_{\lambda,\theta}(t)<0$ for $t\in (T_*,\infty)$, where $T_*=T_*(\theta)>0$ is set as
\begin{align*}
	T_*:=\left[\frac{a_0(p-m)}{\theta b_0(q^*-m)}\right]^{\frac{1}{q^*-p}}>0,
\end{align*}
 Thus, if $\widetilde{f}_{\lambda,\theta}(T_*)>0$, then there exist $T_1$, $T_2>0$ such that
\begin{align*}
	T_1<T_*< T_2\quad \text{and}\quad
		\widetilde{f}_{\lambda,\theta}(t)\begin{cases}
			<0,\quad t\in (0,T_1)\cup (T_2,\infty),\\
			>0,\quad t\in (T_1,T_2).
		\end{cases}
	\end{align*}
For this, we observe that
\begin{align*}
	\widetilde{f}_{\lambda,\theta}(T_*)=-\lambda^{\frac{1}{\delta}}
	+a_0^{\frac{q^*-m}{q^*-p}}b_0^{\frac{m-p}{q^*-p}}\left(\frac{p-m}{q^*-m}\right)^{\frac{p-m}{q^*-p}}\frac{q^*-p}{q^*-m}\theta^{\frac{m-p}{q^*-p}}>0,
\end{align*}
if we assume $\lambda<\widetilde{\lambda}$, with $\widetilde{\lambda}=\widetilde{\lambda}(\theta)$ given as
\begin{align}\label{PT2.la0}
	\widetilde{\lambda}:=c_0^\delta\,\theta^{\frac{\delta(m-p)}{q^*-p}},
\end{align}
where
\begin{align*}
	c_0:=a_0^{\frac{q^*-m}{q^*-p}}b_0^{\frac{m-p}{q^*-p}}\left(\frac{p-m}{q^*-m}\right)^{\frac{p-m}{q^*-p}}\frac{q^*-p}{q^*-m}>0.
\end{align*}
Furthermore, from $\widetilde{f}_{\lambda,\theta}(T_1)=0$ we easily get
\begin{align}\label{PT2.T1}
	T_1>\left(a_0^{-\delta}\lambda\right)^{\frac{1}{\delta(p-m)}}
\end{align}
and we observe that $T_1>1$ if $\lambda>a_0^\delta$. For this, we need
\begin{align*}
	\widetilde{\lambda}>a_0^\delta,
\end{align*}
which holds true whenever $\theta\in \left(0,\theta_0\right)$ with
\begin{align}\label{PT2.the0}
	\theta_0:=(c_0a_0^{-1})^{\frac{q^*-p}{p-m}}.
\end{align}
Thus, from the analysis above, if we consider $\theta\in\left(0,\theta_0\right)$ and $\lambda\in\left(a_0^\delta,\widetilde{\lambda}\right)$, then
$f_{\lambda,\theta}(T_1)=f_{\lambda,\theta}(T_2)=0$, with $1<T_1<T_*<T_2$, and
\begin{align}\label{PT2.E3}
	f_{\lambda,\theta}(t)\begin{cases}
		<0,\quad t\in (0,T_1)\cup (T_2,\infty),\\
		>0,\quad t\in (T_1,T_2).
	\end{cases}
\end{align}
Let $\theta\in \left(0,\theta_0\right)$ and let $\lambda\in \left(a_0^\delta, \widetilde{\lambda}\right)$. Then, from now on, we take $t_1=T_1$ and $t_2=T_2$ in \eqref{cutoff}, so that by \eqref{jphi} we have
\begin{align}
	{J_\phi}(u)&\geq J(u),\quad\text{for any } u\in W_0^{1,\mathcal{H}}(\Omega),\label{PT2.Ph1}\\
	{J_\phi}(u)&=J(u),\quad\text{for any } u\in W_0^{1,\mathcal{H}}(\Omega) \text{ with } \|u\|\leq T_1,\label{PT2.Ph2}
\end{align}
and
\begin{align}
	\label{PT2.Ph3}
	{J_\phi}(u)=\int_\Omega \mathcal{A}(x,\nabla u) \,\mathrm{d} x-\frac{\lambda}{s}\int_\Omega w(x)|u|^s
	\,\mathrm{d} x
\end{align}
for any $ u\in W_0^{1,\mathcal{H}}(\Omega)$ with $\|u\|\geq T_2$.

This ensures that we can always return to $J$ when $J_\phi$ reaches negative values, as shown in the next lemma.

\begin{lemma}\label{LPhi<0}
	Let hypotheses \eqref{H1}--\eqref{H3} be satisfied, let $\theta\in \left(0,\theta_0\right)$ and $\lambda\in \left(a_0^\delta,\widetilde{\lambda}\right)$, with $\delta$, $a_0$, $\widetilde{\lambda}$ and $\theta_0$ as defined in \eqref{delta}, \eqref{a0}, \eqref{PT2.la0} and \eqref{PT2.the0}, respectively. Then, ${J_\phi}(u)<0$ implies that $\|u\|<T_1$, and so ${J_\phi}(u)=J(u)$ and ${J_\phi}'(u)=J'(u)$.
\end{lemma}

\begin{proof}
	Fix $\theta\in \left(0,\theta_0\right)$, $\lambda\in \left(a_0^\delta,\widetilde{\lambda}\right)$ and let ${J_\phi}(u)<0$. Thus, $J(u)<0$ due to \eqref{PT2.Ph1}. Suppose by contradiction that $\|u\|\geq T_1>1$. Then, it follows from \eqref{PTcc.gi} that $f_{\lambda,\theta}(\|u\|)<0$. By \eqref{PT2.E3}, we have $\|u\|>T_2>1$. Hence, we conclude from \eqref{PT2.Ph3} that
	\begin{align*}
		J_\phi(u)=\int_\Omega \mathcal{A}(x,\nabla u) \,\mathrm{d} x-\frac{\lambda}{s}\int_\Omega w(x)|u|^s\,\mathrm{d} x<0.
	\end{align*}
	From this, using \eqref{H3}, Proposition \ref{P.m-n} and \eqref{PT2.E2}, we obtain
	\begin{align*}
		\frac{1}{2q}\|u\|^{p}-k_1\lambda^{\frac{1}{\delta}}\|u\|^{m}<0,
	\end{align*}
	with $k_1$ as given in \eqref{k1}. This along with \eqref{a0} yields
	\begin{align*}
		T_1\leq \|u\|\leq \left(2qk_1\lambda^{\frac{1}{\delta}}\right)^{\frac{1}{p-m}}=\left(a_0^{-\delta}\lambda\right)^{\frac{1}{\delta(p-m)}},
	\end{align*}
	which contradicts \eqref{PT2.T1}. Thus, $\|u\|<T_1$ and so ${J_\phi}(u)=J(u)$ and ${J_\phi}'(u)=J'(u)$ due to \eqref{PT2.Ph2}. This shows the assertion.
\end{proof}

Considering Lemmas \ref{LPhi<0} and \ref{Le.PS}, the validity of the compactness condition of ${J_\phi}$ for negative levels can be established if the positivity of the right-hand side of \eqref{PS_c2} (or \eqref{PS_c2'} if $b\equiv 0$) is guaranteed.	To this end, we observe that for $\theta\in (0,1)$, we can rewrite the right-hand side of \eqref{PS_c2} as
\begin{align*}
	\overline{c}_{\lambda,\theta} & :=C_1\theta^{-\frac{q}{q^*-q}}- C_3\lambda^{\frac{q^\ast}{q^\ast-s}}\theta^{-\frac{s}{q^\ast-s}}-C_2\lambda^{\frac{q}{q-s}} \\
	& =\theta^{-\frac{q}{q^*-q}}\left[\frac{1}{2}C_1 -C_3\lambda^{\frac{q^\ast}{q^\ast-s}}\theta^{\frac{q^\ast(q-s)}{(q^\ast-q)(q^\ast-s)}}\right]+\frac{1}{2}C_1\theta^{-\frac{q}{q^*-q}}-C_2\lambda^{\frac{q}{q-s}}.
\end{align*}
Observe that $\overline{c}_{\lambda,\theta}>0$ provided that
\begin{equation}
	\label{PT2.bala}
	\lambda<C_5\theta^{-\frac{q-s}{q^*-q}}=:\overline{\lambda}
\end{equation}
with
\begin{align*}
	C_5:=\min\left\{\left(\frac{1}{2}C_1C_3^{-1}\right)^{\frac{q^*-s}{q^*}},\left(\frac{1}{2}C_1C_2^{-1}\right)^{\frac{q-s}{q}}\right\}.
\end{align*}
For the case $b\equiv 0$, we just take $\overline{c}_{\lambda,\theta}  :=C_1\theta^{-\frac{p}{p^*-p}} -C_2\lambda^{\frac{q}{q-s}}$ and  $C_5:=\left(C_1C_2^{-1}\right)^{\frac{q-s}{q}}$. Thus, we set
\begin{align}\label{lambdastar}
	\lambda^*:=\min\left\{\widetilde{\lambda},\overline{\lambda}\right\}.
\end{align}

\begin{lemma}\label{LPhiPS}
	Let hypotheses \eqref{H1}--\eqref{H3} be satisfied, and let $\theta\in \left(0,\min\{\theta_0,1\}\right)$ satisfy $a_0^\delta<\lambda^*$, where $a_0$, $\theta_0$  and $\lambda_*$ are given in  \eqref{a0}, \eqref{PT2.the0} and \eqref{lambdastar}, respectively. Then, for any $\lambda\in \left(a_0^\delta,\lambda^*\right)$, the functional ${J_\phi}$ satisfies the $\textup{(PS)}_c$ condition for any $c<0$.
\end{lemma}

\begin{proof}
	Fix $\theta\in \left(0,\min\{\theta_0,1\}\right)$ such that $a_0^\delta<\lambda^*$. Let $\lambda\in \left(a_0^\delta,\lambda^*\right)$, and let $\{u_n\}_{n\in\mathbb{N}}$ be a $\textup{(PS)}_c$ sequence for the functional ${J_\phi}$ with $c<0$, that is, \eqref{ps} with $I=J_\phi$ holds.
	Then, there exists $n_{0}\in\mathbb N$ large enough such that ${J_\phi}(u_{n})<0$ for any $n>n_{0}$. Consequently, by Lemma \ref{LPhi<0} we get $\|u_{n}\|<T_1$, and so ${J_\phi}(u_{n})=J(u_{n})$ and ${J_\phi}'(u_{n})=J'(u_{n})$ for $n>n_{0}$. This fact implies that $\{u_n\}_{n\in\mathbb{N}}$ is a bounded $\textup{(PS)}_c$ sequence for the functional $J$. By the choice of $\lambda^*$ as given in \eqref{lambdastar}, the sequence  $\{u_n\}_{n\in\mathbb{N}}$ admits a convergent subsequence in $W_0^{1,\mathcal{H}}(\Omega)$ in view of Lemma \ref{Le.PS}. The proof is complete.
\end{proof}

Summarizing, we observe that, once $j\in\N$ is fixed,  the functional $J_\phi$ verifies Lemmas \ref{Le.c_n<0}--\ref{LPhiPS} if $\theta\in\left(0,\min\left\{1,\theta_0\right\}\right)$  satisfying $\max\left\{a_0^\delta,d_j\right\}<\lambda^*$ and if $\lambda\in\left(\max\left\{a_0^\delta,d_j\right\},\lambda^*\right)$. Hence, we need to guarantee that
\begin{align}\label{provare}
	\max\left\{a_0^\delta,d_j\right\}<\lambda^*,\quad\text{when }\theta<\min\left\{1,\theta_0\right\}.
\end{align}
In this direction, from the sequence $\{\delta_j\}_{j\in\mathbb{N}}$ given in \eqref{equi.norms}, we define $\{\lambda_j\}_{j\in\mathbb{N}}$ as
\begin{align}\label{PT2.lak}
	\lambda_j:=C_6\delta_j^{2s},\quad\text{for any } j\in\N,
\end{align}
where $C_6:=a_0^\delta+C_4$ with $a_0$ and $C_4$ as given in \eqref{a0} and \eqref{PT2.dk}, respectively. Since $1<\delta_j<\delta_{j+1}$ for any $j\in\N$, we have
\begin{align}\label{monotonia}
	\max\left\{a_0^\delta,d_j\right\}<\lambda_j<\lambda_{j+1},\quad\text{for any }  j\in\N.
\end{align}
Note that with $\widetilde{\lambda}$ given by \eqref{PT2.la0}, we have
\begin{align*}
	\lambda_j=\widetilde{\lambda}=c_0^\delta\theta^{\frac{\delta(m-p)}{q^*-p}}\quad \Longleftrightarrow \quad  \theta=\left(c_0^{-\delta}C_6\right)^{\frac{q^*-p}{\delta(m-p)}}\delta_j^{-\frac{2s(q^*-p)}{\delta(p-m)}}.
\end{align*}
We also observe that considering $\overline{\lambda}$ given in \eqref{PT2.bala}, it follows that
\begin{align*}
	\lambda_j=\overline{\lambda}=C_5\theta^{-\frac{q-s}{q^*-q}} \quad \Longleftrightarrow \quad  \theta=\left(C_5^{-1}C_6\right)^{\frac{q-q^*}{q-s}}\delta_j^{-\frac{2s(q^*-q)}{q-s}}.
\end{align*}
Thus, let us set
\begin{align}\label{PT2.thek}
	\theta_j:=\min\left\{1,\theta_0,\left(c_0^{-\delta}C_6\right)^{\frac{q^*-p}{\delta(m-p)}},\left(C_5^{-1}C_6\right)^{\frac{q-q^*}{q-s}}\right\}\delta_j^{-\kappa}
\end{align}
with $\kappa:=\max\left\{\frac{2s(q^*-p)}{\delta(p-m)},\frac{2s(q^*-q)}{q-s}\right\}$. We derive from \eqref{PT2.la0} and \eqref{PT2.bala} that $\theta_j>0$ is independent of $\lambda$ and for any $\theta\in \left(0,\theta_j\right)$, we have
\begin{align*}
	\theta<\min\left\{1,\theta_0\right\} \quad \text{and}\quad \lambda_j<\lambda^*,
\end{align*}
with $\lambda^*$ as given in \eqref{lambdastar}.
Hence, considering also \eqref{monotonia}, we have the validity of \eqref{provare} whenever $\theta\in \left(0,\theta_j\right)$. Now, we are in a position to prove Theorem~\ref{Theo.MS}.

\begin{proof}[Proof of Theorem \ref{Theo.MS}]
	Let $j\in\N$ be fixed and let us set $\lambda_j$ and $\theta_j$ as in \eqref{PT2.lak} and \eqref{PT2.thek}, respectively. Let $\theta\in \left(0,\theta_j\right)$, so that $\lambda_j<\lambda^*$, with $\lambda^*$ given in \eqref{lambdastar}. Then, let $\lambda\in\left(\lambda_j,\lambda^*\right)$ and let us consider the minimax sequence $\{c_j\}_{j\in\mathbb{N}}$ given in \eqref{cn}. By Lemma \ref{Le.c_n<0}, we know that
	\begin{align*}
		-\infty<c_i<0,\quad\text{for any }\,i=1,\ldots,j.
	\end{align*}
	Thus, Lemma \ref{LPhiPS} yields that $J_\phi$ satisfies the $\textup{(PS)}_{c_i}$ condition, for any $i=1,\ldots,j$ and so,  $c_i$ are critical values for $J_\phi$ (see Rabinowitz \cite{Rabinowitz-1986} for the details). Setting
	\begin{align*}
		K_c:=\left\{u\in W_0^{1,\mathcal{H}}(\Omega)\setminus\{0\}\colon\,\, J_\phi(u)=c\text{ and }J_\phi'(u)=0\right\},
	\end{align*}
	we infer that $K_{c_i}$ are compact, for any $i=1,\ldots,j$.

	Now, we distinguish two situations. Either $\{c_i\colon\,\, i=1,\ldots,j\}$ are $j$ distinct critical values of $J_\phi$ or $c_n=c_{n+1}=\ldots=c_j=c$ for some $n\in\{1,\ldots,j-1\}$. In the second situation, since $K_c$ is compact, by a deformation lemma, standard properties of the genus (see again Rabinowitz \cite[Proposition 7.7]{Rabinowitz-1986}) and arguing as in Farkas--Fiscella-Winkert \cite[Lemma 3.6]{Farkas-Fiscella-Winkert-2022}, we get
	\begin{align*}
		\gamma(K_c)\geq j-n+1\geq2.
	\end{align*}
	Thus, $K_c$ has infinitely many points, see Rabinowitz \cite[Remark 7.3]{Rabinowitz-1986}, which are infinitely many critical values for $J_\phi$. Consequently, $J_\phi$ admits at least $j$ negative critical values, which represent at least $j$ negative critical values for $J$ thanks to Lemma \ref{LPhi<0}.
\end{proof}

\section*{Acknowledgments}
C. Farkas was supported by the J\'anos Bolyai Research Scholarship of the Hungarian Academy of Science. A. Fiscella is member of {\em Gruppo Nazionale per l'Analisi Ma\-te\-ma\-ti\-ca, la Probabilit\`a e le loro Applicazioni} (GNAMPA) of the {\em Istituto Nazionale di Alta Matematica} (INdAM). A. Fiscella realized the manuscript within the auspices of the FAPESP project titled \textit{Non-uniformly elliptic problems} (2024/04156-0). K. Ho was supported by the University of Economics Ho Chi Minh City (UEH), Vietnam.


\begin{thebibliography}{99}

\bibitem{Arora-Fiscella-Mukherjee-Winkert-2022}
	R. Arora, A. Fiscella, T. Mukherjee, P. Winkert,
	{\it On critical double phase Kirchhoff problems with singular nonlinearity},
	Rend. Circ. Mat. Palermo (2) {\bf 71} (2022), no. 3, 1079--1106.

\bibitem{Bahrouni-Radulescu-Repovs-2019}
	A. Bahrouni, V.D. R\u{a}dulescu, D.D. Repov\v{s},
	{\it Double phase transonic flow problems with variable growth: nonlinear patterns and stationary waves},
	Nonlinearity {\bf 32} (2019), no. 7, 2481--2495.

\bibitem{Baldelli-Brizi-Filippucci-2021}
	L. Baldelli, Y. Brizi, R. Filippucci,
	{\it Multiplicity results for $(p,q)$-Laplacian equations with critical exponent in $\mathbb{R}^N$ and negative energy},
	Calc. Var. Partial Differential Equations {\bf 60} (2021), no. 1, Paper No. 8, 30 pp.

\bibitem{Baroni-Colombo-Mingione-2015}
	P. Baroni, M. Colombo, G. Mingione,
	{\it Harnack inequalities for double phase functionals},
	Nonlinear Anal. {\bf 121} (2015), 206--222.

\bibitem{Baroni-Colombo-Mingione-2016}
	P. Baroni, M. Colombo, G. Mingione,
	{\it Non-autonomous functionals, borderline cases and related function classes},
	St. Petersburg Math. J. {\bf 27} (2016), 347--379.

\bibitem{Baroni-Colombo-Mingione-2018}
	P. Baroni, M. Colombo, G. Mingione,
	{\it Regularity for general functionals with double phase},
	Calc. Var. Partial Differential Equations {\bf 57} (2018), no. 2, Art. 62, 48 pp.

\bibitem{Benci-DAvenia-Fortunato-Pisani-2000}
	V. Benci, P. D'Avenia, D. Fortunato, L. Pisani,
	{\it Solitons in several space dimensions: Derrick's problem and  infinitely many solutions},
	Arch. Ration. Mech. Anal. {\bf 154} (2000), no. 4, 297--324.

\bibitem{Brezis-Nirenberg-1983}
	H. Br\'{e}zis, L. Nirenberg,
	{\it Positive solutions of nonlinear elliptic equations involving critical Sobolev exponents},
	Comm. Pure Appl. Math. {\bf 36} (1983), no. 4, 437--477.

\bibitem{Cherfils-Ilyasov-2005}
	L. Cherfils, Y. Il'yasov,
	{\it On the stationary solutions of generalized reaction diffusion  equations with {$p\&q$}-{L}aplacian},
	Commun. Pure Appl. Anal. {\bf 4} (2005), no. 1, 9--22.

\bibitem{Chlebicka-Gwiazda-Swierczewska-Gwiazda-Wroblewska-Kaminska-2021}
	I. Chlebicka, P. Gwiazda, A. \'{S}wierczewska-Gwiazda, A. Wr\'{o}blewska-Kami\'{n}ska,
	``Partial Differential Equations in Anisotropic Musielak-Orlicz Spaces'',
	Springer, Cham, 2021.

\bibitem{Colasuonno-Perera-2025}
	F. Colasuonno, K. Perera,
	{\it Critical growth double phase problems: The local case and a Kirchhoff type case},
	J. Differential Equations {\bf 422} (2025), 426--488.

\bibitem{Colasuonno-Squassina-2016}
	F. Colasuonno, M. Squassina,
	{\it Eigenvalues for double phase variational integrals},
	Ann. Mat. Pura Appl. (4) {\bf 195} (2016), no. 6, 1917--1959.

\bibitem{Colombo-Mingione-2015a}
	M. Colombo, G. Mingione,
	{\it Bounded minimisers of double phase variational integrals},
	Arch. Ration. Mech. Anal. {\bf 218} (2015), no. 1, 219--273.

\bibitem{Colombo-Mingione-2015b}
	M. Colombo, G. Mingione,
	{\it Regularity for double phase variational problems},
	Arch. Ration. Mech. Anal. {\bf 215} (2015), no. 2, 443--496.

\bibitem{Crespo-Blanco-Gasinski-Harjulehto-Winkert-2022}
	\'{A}. Crespo-Blanco, L. Gasi\'nski, P. Harjulehto, P. Winkert,
	{\it A new class of double phase variable exponent problems: Existence and uniqueness},
	J. Differential Equations {\bf 323} (2022), 182--228.

\bibitem{Cupini-Marcellini-Mascolo-2023}
	G. Cupini, P. Marcellini, E. Mascolo,
	{\it Local boundedness of weak solutions to elliptic equations with $p,q$-growth},
	Math. Eng. {\bf 5} (2023), no. 3, Paper No. 065, 28 pp.

\bibitem{Diening-Harjulehto-Hasto-Ruzicka-2011}
	L. Diening, P. Harjulehto, P. H\"{a}st\"{o}, M. R\r{u}\v{z}i\v{c}ka,
	``Lebesgue and Sobolev Spaces with Variable Exponents'',
	Springer, Heidelberg, 2011.

\bibitem{Farkas-Fiscella-Winkert-2022}
	C. Farkas, A. Fiscella, P. Winkert,
	{\it On a class of critical double phase problems},
	J. Math. Anal. Appl. {\bf 515} (2022), no. 2, Paper No. 126420, 16 pp.

\bibitem{Farkas-Winkert-2021}
	C. Farkas, P. Winkert,
	{\it An existence result for singular Finsler double phase problems},
	J. Differential Equations {\bf 286} (2021), 455--473.

\bibitem{Feng-Bai-2024}
	Y. Feng, Z. Bai,
	{\it Multiple nontrivial solutions for a double phase system with concave-convex nonlinearities in subcritical and critical cases},
	Anal. Math. Phys. {\bf 14} (2024), no. 6, Paper No. 125, 44 pp.

\bibitem{Fonseca-Leoni-2007}
	I. Fonseca, G. Leoni,
	``Modern Methods in the Calculus of Variations: $L^p$ Spaces'',
	Springer, New York, 2007.

\bibitem{Ha-Ho-2024}
	H.H. Ha, K. Ho,
	{\it Multiplicity results for double phase problems involving a new type of critical growth},
	J. Math. Anal. Appl. {\bf 530} (2024), no. 1, Paper No. 127659, 36 pp.

\bibitem{Harjulehto-Hasto-2019}
	P. Harjulehto, P. H\"{a}st\"{o},
	``Orlicz Spaces and Generalized Orlicz Spaces'',
	Springer, Cham, 2019.

\bibitem{Ho-Kim-Sim-2019}
	K. Ho, Y.-H. Kim, I. Sim,
	{\it Existence results for Schr\"{o}dinger $p(\cdot)$-Laplace equations involving critical growth in $\mathbb{R}^N$},
	Nonlinear Anal. {\bf 182} (2019), 20--44.

\bibitem{Ho-Kim-Zhang-2024}
	K. Ho, Y.-H. Kim, C. Zhang,
	{\it Double phase anisotropic variational problems involving critical growth},
	Adv. Nonlinear Anal. {\bf 13} (2024), no. 1, Paper No. 20240010, 38 pp.

\bibitem{Ho-Perera-Sim-2023}
	K. Ho, K. Perera, I. Sim,
	{\it On the Brezis-Nirenberg problem for the $(p,q)$-Laplacian},
	Ann. Mat. Pura Appl. (4) {\bf 202} (2023), no. 4, 1991--2005.

\bibitem{Ho-Sim-2021}
	K. Ho, I. Sim,
	{\it An existence result for $(p,q)$-Laplace equations involving sandwich-type and critical growth},
	Appl. Math. Lett. {\bf 111} (2021), Paper No. 106646, 8 pp.

\bibitem{Ho-Sim-2016}
	K. Ho, I. Sim,
	{\it On degenerate $p(x)$-Laplace equations involving critical growth with two parameters},
	Nonlinear Anal. {\bf 132} (2016), 95--114.

\bibitem{Ho-Sim-2023}
	K. Ho, I. Sim,
	{\it On sufficient ``local'' conditions for existence results to generalized $p(\cdot)$-Laplace equations involving critical growth},
	Adv. Nonlinear Anal. {\bf 12} (2023), no. 1, 182--209.

\bibitem{Ho-Winkert-2023}
	K. Ho, P. Winkert,
	{\it New embedding results for double phase problems with variable exponents and a priori bounds for corresponding generalized double phase problems},
	Calc. Var. Partial Differential Equations {\bf 62} (2023), no. 8, Paper No. 227, 38 pp.

\bibitem{Kawohl-Lucia-Prashanth-2007}
	B. Kawohl, M. Lucia, S. Prashanth,
	{\it Simplicity of the principal eigenvalue for indefinite quasilinear problems},
	Adv. Differential Equations {\bf 12} (2007), no. 4, 407--434.

\bibitem{Kumar-Radulescu-Sreenadh-2020}
	D. Kumar, V.D. R\u{a}dulescu, K. Sreenadh,
	{\it Singular elliptic problems with unbalanced growth and critical exponent},
	Nonlinearity {\bf 33} (2020), no. 7, 3336--3369.

\bibitem{Lions-1985}
	P.-L. Lions,
	{\it The concentration-compactness principle in the calculus of variations. The limit case. I},
	Rev. Mat. Iberoamericana {\bf 1} (1985), no. 1, 145--201.

\bibitem{Liu-Dai-2018}
	W. Liu, G. Dai,
	{\it Existence and multiplicity results for double phase problem},
	J. Differential Equations {\bf 265} (2018), no. 9, 4311--4334.

\bibitem{Liu-Papageorgiou-2021}
	Z. Liu, N.S. Papageorgiou,
	{\it Asymptotically vanishing nodal solutions for critical double phase problems},
	Asymptot. Anal. {\bf 124} (2021), no. 3-4, 291--302.

\bibitem{Marcellini-2021}
	P. Marcellini,
	{\it Growth conditions and regularity for weak solutions to nonlinear elliptic pdes},
	J. Math. Anal. Appl. {\bf 501} (2021), no. 1, Paper No. 124408, 32 pp.

\bibitem{Marcellini-2023}
	P. Marcellini,
	{\it  Local Lipschitz continuity for $p,q$-PDEs with explicit $u$-dependence},
	Nonlinear Anal. {\bf 226} (2023), Paper No. 113066, 26 pp.

\bibitem{Marcellini-1991}
	P. Marcellini,
	{\it Regularity and existence of solutions of elliptic equations with $p,q$-growth conditions},
	J. Differential Equations {\bf 90} (1991), no. 1, 1--30.

\bibitem{Marcellini-1989}
	P. Marcellini,
	{\it Regularity of minimizers of integrals of the calculus of variations with nonstandard growth conditions},
	Arch. Rational Mech. Anal. {\bf 105} (1989), no. 3,  267--284.

\bibitem{Papageorgiou-Vetro-Winkert-2024}
	N.S. Papageorgiou, F. Vetro, P. Winkert,
	{\it Sequences of nodal solutions for critical double phase problems with variable exponents},
	Z. Angew. Math. Phys. {\bf 75} (2024), no. 3, Paper No. 95, 17 pp.

\bibitem{Papageorgiou-Vetro-Winkert-2023}
	N.S. Papageorgiou, F. Vetro, P. Winkert,
	{\it Sign changing solutions for critical double phase problems with variable exponent},
	Z. Anal. Anwend. {\bf 42} (2023), no. 1-2, 235--251.

\bibitem{Papageorgiou-Winkert-2024}
	N.S. Papageorgiou, P. Winkert,
	``Applied Nonlinear Functional Analysis'',
	Second revised edition,  De Gruyter, Berlin, 2024.

\bibitem{Papageorgiou-Zhang-2020}
	N.S. Papageorgiou, C. Zhang,
	{\it Double phase problem with critical and locally defined reaction terms},
	Asymptot. Anal. {\bf 116} (2020), no. 2, 73--92.

\bibitem{Rabinowitz-1986}
	P.H. Rabinowitz,
	``Minimax Methods in Critical Point Theory with Applications to Differential Equations'',
	American Mathematical Society, Providence, RI, 1986.

\bibitem{Zhikov-1986}
	V.V. Zhikov,
	{\it Averaging of functionals of the calculus of variations and elasticity theory},
	Izv. Akad. Nauk SSSR Ser. Mat. {\bf 50} (1986), no. 4, 675--710.

\bibitem{Zhikov-1995}
	V.V. Zhikov,
	{\it On Lavrentiev's phenomenon},
	Russian J. Math. Phys. {\bf 3} (1995), no. 2, 249--269.

\bibitem{Zhikov-2011}
	V.V. Zhikov,
	{\it On variational problems and nonlinear elliptic equations with nonstandard growth conditions},
	J. Math. Sci. {\bf 173} (2011), no. 5, 463--570.

\end{thebibliography}
\end{document}